\documentclass[11pt,letterpaper]{amsart}
\usepackage{preamble}
\counterwithin{equation}{section}
\begin{document}

\title[Monotonicity, Topology, and Convexity of Recurrence in Random Walks]{Monotonicity, Topology, and Convexity of Recurrence in Random Walks}

\author[Rupert Li]{Rupert Li}
\address[]{Stanford University, Stanford, CA 94305, USA}
\email{rupertli@stanford.edu}

\author[Elchanan Mossel]{Elchanan Mossel}
\address[]{Massachusetts Institute of Technology, Cambridge, MA 02139, USA}
\email{elmos@mit.edu}

\author[Benjamin Weiss]{Benjamin Weiss}
\address[]{Hebrew University of Jerusalem, Jerusalem, Israel} \email{weiss@math.huji.ac.il}

\begin{abstract}
We consider non-homogeneous random walks on the two-dimensional positive quadrant $\N^2$ and the one-dimensional slab $\{0,1,\dots,k\}\times\N$.
In the 1960's the following question was asked for $\N^2$: is it true if such a random walk $X$ is recurrent and $Y$ is another random walk that at every point is more likely to go down and more likely to go left than $Y$, then $Y$ is also recurrent? 

We provide an example showing that the answer is negative.
We also show, via a coupling argument, that if either the random walk $X$ or $Y$ is sufficiently homogeneous then the answer is in fact positive.
In addition, we show using the Rayleigh monotonicity principle that the analogous question for random walks on trees is positive.

These results show that the subset of parameter space that yields recurrent random walks possesses some geometric properties, in this case the structure of an order ideal.
Motivated by this perspective, we consider the more symmetric setting of homogeneous random walks on finitely generated abelian groups, and ask when this subset possesses other geometric properties, namely various topological properties and convexity.
We answer some of these questions: in particular, we show that this subset is closed, and under a symmetric support condition, show it is path-connected and additionally show it is convex if and only if its effective dimension is at most 2.
We also show its complement is in some sense typically path-connected but not convex.
We finally propose some related open problems.  
\end{abstract}

\maketitle

\section{Introduction}\label{section:introduction}
For the non-homogeneous reflecting random walk on $\N$ with transitions $p_n$ from $n$ to $n+1$ and 
$q_n = 1-p_n$ from $n$ to $n-1$ (for $n>0$) there is an explicit characterization of when the 
random walk is recurrent/transient (see \cite[Section XV.8]{Feller}). From this it follows directly 
that if the random walk is recurrent and we consider a different one in which for all $n$ the $p_n'$ 
satisfy $p_n' \leq p_n$ then the new random walk is also recurrent, while if the first was transient 
while the second satisfies $p_n' \geq p_n$ for all $n$ then the second is also transient. These 
monotonicity properties can also be easily seen via a coupling argument.
  
In the early 1960's the question as to whether or not a similar monotonicity property holds for 
two-dimensional random walks circulated among the probability group at Princeton University.
We are not sure if this is the first time this question was asked.
To the best of our knowledge this question has not been resolved to date, though it is not hard to derive the answer using a 1972 result by Malyshev \cite{MalyshevRussian}; see \cref{remark:Malyshev}.
The purpose of the following note is to resolve this question and related questions.
    
We consider non-homogeneous random walks on the positive quadrant in two dimensions. Such a random walk is a Markov chain where the state space is $\N^2$ and one may only move to one of the four states immediately left, right, up, or down from the current state. 
We refer readers to \cite{Spitzer} for a broad overview of random walks.

Our main interest is in understanding if recurrence is a monotone property with respect to the following partial order: 
\begin{definition}\label{definition:N2_partial_order}
For two random walks $X$ and $Y$ on $\N^2$, say $X \preceq Y$ if for each state, the probability of going down in $X$ is greater than or equal to that in $Y$, the probability of going left in $X$ is greater than or equal to that in $Y$, the probability of going up in $Y$ is greater than or equal to that in $X$, and the probability of going right in $Y$ is greater than or equal to that in $X$.
\end{definition}
Intuitively, $X\preceq Y$ means $X$ goes down and left more than $Y$.

The main question we address in the paper is the following:
\begin{question} \label{q:main}
Let $X$ and $Y$ be two non-homogeneous irreducible random walks on $\N^2$ with $X\preceq Y$. Is it true that if $Y$ is recurrent then so is $X$? 
In other words is it true that if $X$ is transient then so is $Y$?
\end{question}

Note that for two irreducible Markov chains $X$ and $Y$ on the same state space, if there only exist finitely many states for which the outgoing transition probabilities are different for $X$ and $Y$, i.e., $X$ and $Y$ agree on a co-finite set of states, then they will have the same recurrence behavior: $X$ is recurrent, positive recurrent, or transient if and only if $Y$ is recurrent, positive recurrent, or transient, respectively.
Thus, when considering our question of monotonicity of recurrence, we can generalize Definition~\ref{definition:N2_partial_order}, as well as all other definitions that follow in this paper, so that the required inequalities only need to hold on a co-finite set of states.
For example, we can ignore any inequalities concerning the transition probabilities from the origin.
However, for sake of clarity and brevity, we henceforth do not discuss this equivalence up to a co-finite set.

The intuition behind \cref{q:main} is that if $Y$ is recurrent then it returns to the origin with probability $1$.
As $X \preceq Y$ it is natural to expect that $X$ will also return to the origin with probability $1$, as indeed is the case in nearest-neighbor walks on $\N$ as previously discussed.

In our main result we prove that, in fact, the answer is negative.

\begin{theorem} \label{thm:ex}
There exists two non-homogeneous random walks $X\preceq Y$ such that $Y$ is positive recurrent and $X$ is transient. Moreover $X$ and $Y$ are elliptic, meaning all possible transitions have positive probabilities. 
\end{theorem}
Our construction (see \cref{fig:schematic} for a schematic) is of a similar nature as the following perhaps more intuitive example of a non-nearest-neighbor walk on $\N$, for which we only sketch the main idea.
Suppose we have a random walk $Y$ on $\N$ whose jumps take values in $\set{-2,-1,1,2}$, and whose odd states drift towards $0$ and even states drift towards $\infty$, in such a way that the walk tends to stay on odd states enough that $Y$ is positive recurrent.
Then, one can make each state drift slightly more towards the left but in such a way that the walk tends to stay on the even states, so that the new Markov chain $X$ is transient.
Our construction essentially uses the positive $x$-axis as the escape route, analogous to the even states, while the positive $y$-axis plays the role of the odd states.

On the other hand, we find conditions under which recurrence is monotone with respect to $\preceq$.

\begin{definition}\label{definition:inward_homogeneous}
Say a random walk on $\N^2$ is \textit{inward-homogeneous} if it is of the following form:
\begin{equation}\label{equation:inward_homogeneous}
    \begin{cases} \text{right} & \text{w.p. } r_o \\ \text{up} & \text{w.p. } u_o, \end{cases} \qquad
    \begin{cases} \text{left} & \text{w.p. } \ell_x \\ \text{right} & \text{w.p. } r_x \\ \text{up} & \text{w.p. } u_x, \end{cases} \qquad
    \begin{cases} \text{down} & \text{w.p. } d_y \\ \text{up} & \text{w.p. } u_y \\ \text{right} & \text{w.p. } r_y, \end{cases} \qquad
    \begin{cases} \text{left} & \text{w.p. } \ell_q \\ \text{right} & \text{w.p. } r_q \\ \text{up} & \text{w.p. } u_q \\\text{down} & \text{w.p. } d_q, \end{cases}
\end{equation}
for $(0,0)$, $\{(i,0):i\geq1\}$, $\{(0,j):j\geq1\}$, and $\{(i,j):i,j\geq1\}$, respectively, where
\begin{equation}\label{equation:inward_homogeneous_conditions}
    r_o \geq r_x = r_q, \quad \ell_x \geq \ell_q, \quad u_x \geq u_q, \quad u_o \geq u_y = u_q, \quad d_y \geq d_q, \quad\text{and}\quad r_y \geq r_q.
\end{equation}
A random walk on $\N^2$ of the above form \eqref{equation:inward_homogeneous}, though not necessarily satisfying \eqref{equation:inward_homogeneous_conditions}, is commonly referred to as \textit{maximally homogeneous} in the literature, and is the one most studied in the literature for random walks on the quadrant; see, e.g., \cite{FMM1,FMM2,FIM}.
Typically, maximal homogeneity is defined to only require the mean drifts to be the same in the four regions as opposed to requiring the transition probabilities to be the same.
\end{definition}
One can view an inward-homogeneous random walk as being given by the transition probabilities $\ell_q$, $r_q$, $u_q$, and $d_q$ from the general quadrant case, along with three stochastic ``redirection rules'' if a forbidden move is selected, i.e., moving down while on the positive $x$-axis, moving left while on the positive $y$-axis, or moving left or down while at the origin.
The only condition for inward-homogeneity is that the positive $x$-axis cannot redirect towards the right, and the positive $y$-axis cannot redirect up. 
\begin{remark}\label{remark:queuing}
    Comparison ideas are widely used in the queuing literature; see, e.g., \cite{Stoyan_German} or its English translation \cite{Stoyan}.
    The queuing setup, specifically the Jackson network of two M/M/1 queues, translates to a continuous time random walk on $\N^2$, where movement in this continuous time random walk is dictated by Poisson rates, and boundary conditions are enforced by simply suppressing illegal moves, i.e., waiting until the Poisson clock of a valid move rings.
    Motivated by this option of being lazy instead of redirecting the illegal move, we introduce the following weaker condition in \cref{definition:weakly_inward_homogeneous}.
    Essentially, instead of applying redirection, e.g., redirecting from down to the left or up when on the positive $x$-axis, one can choose to not move, which in our non-lazy random walk setting corresponds to reallocating some of the down probability $d_q$ (in the case of the positive $x$-axis) to the other three directions proportional to $\ell_q$, $u_q$, and $r_q$.
    Because the following definition allows for both suppressing or redirecting the illegal move, we found that introducing the intermediate definition of inward-homogeneity was still instructive for intuition.
\end{remark}
\begin{definition}\label{definition:weakly_inward_homogeneous}
    Using the notation of \cref{definition:inward_homogeneous}, i.e., \eqref{equation:inward_homogeneous}, say a random walk on $\N^2$ is \textit{weakly inward-homogeneous} if
    \begin{equation}\label{equation:weakly_inward_homogeneous}
        \frac{\ell_x}{\ell_q},\frac{u_x}{u_q}\geq\frac{r_x}{r_q}\geq1 \quad\text{and}\quad \frac{d_y}{d_q},\frac{r_y}{r_q}\geq\frac{u_y}{u_q}\geq1.
    \end{equation}
\end{definition}
In our positive result we show that under a weakly inward-homogeneity condition, recurrence and positive recurrence are monotonic properties with respect to $\preceq$.
\begin{theorem}\label{theorem:monotonicity_recurrence}
    If an irreducible weakly inward-homogeneous random walk $X$ on $\N^2$ is recurrent (respectively, positive recurrent), then any irreducible random walk $Y$ on $\N^2$ such that $Y\preceq X$ is also recurrent (respectively, positive recurrent).
\end{theorem}

For the reverse direction, where we want transience of $X$ to imply transience of $Y$ for $Y\succeq X$, we can of course take the contrapositive of \cref{theorem:monotonicity_recurrence}, but this requires $Y$ to be weakly inward-homogeneous.
The following theorem instead assumes $X$ is weakly inward-homogeneous, and its proof is very similar to that of \cref{theorem:monotonicity_recurrence}.
\begin{theorem}\label{theorem:monotonicity_transience}
    If an irreducible weakly inward-homogeneous random walk $X$ on $\N^2$ is transient, then any irreducible random walk $Y$ on $\N^2$ such that $Y\succeq X$ is also transient.
\end{theorem}

Using \cref{theorem:monotonicity_recurrence,theorem:monotonicity_transience} and their contrapositive, we obtain the following result implying monotonicity when neither random walk is weakly inward-homogeneous, though a weakly inward-homogeneous random walk must exist between the two random walks.
\begin{corollary}\label{corollary:sandwich_monotonicity}
    Let $X$, $Y$, and $Z$ be irreducible random walks on $\N^2$ where $Y$ is weakly inward-homogeneous and $X\preceq Y \preceq Z$.
    If $Z$ is recurrent then $X$ is recurrent, and if $X$ is transient then $Z$ is transient.
\end{corollary}
\begin{remark}\label{remark:Foster_Lyapunov}
    The classification of recurrence and transience for random walks in the quadrant $\N^2$ goes back to Kingman~\cite{Kingman} in 1961 and Malyshev \cite{MalyshevBook,MalyshevRussian} in the early 1970s.
    As discussed in further detail in \cref{remark:Malyshev}, this work is able to immediately imply the two walks constructed for the proof of \cref{thm:ex} are respectively positive recurrent and transient, though this work did not comment on the non-monotonicity of recurrence, e.g., by providing an example demonstrating monotonicity fails.
    The Kingman--Malyshev result provides an essentially complete classification in the maximally homogeneous case where the interior drift is non-zero, with the classification depending on the drift vector in the interior and the two angles produced by the drifts at the boundaries.
    Malyshev's hypothesis of bounded increments was relaxed through many works \cite{Fayolle,Rosenkrantz,FMM2,VaninskiiLazareva,Zachary} to only require finite first moments.
    The essentially complete recurrence classification in the case of zero interior drift took much longer, until the early 1990s \cite{AFM,FMM1,FMM2}, and is more complex, depending on the increment covariance matrix in the interior as well as the two boundary reflection angles.

    The recurrence classification in the non-zero interior drift case is simple enough that it should essentially imply the monotonicity result of \cref{theorem:monotonicity_recurrence} in this case; one can obtain a quadratic Lyapunov function for showing positive recurrence, where the level curves are like those in \cite[Figure~3.3.1]{FMM2}.
    The zero interior drift case is more delicate, as we must consider the interior covariance, but more recent work, e.g., \cite{dCMW}, may be able to deduce a similar result by applying a suitable transformation to put the walk into a wedge.
    An advantage of these Foster--Lyapunov methods is that they depend only on the drifts (and sometimes the covariances), providing the potential to expand our results beyond nearest-neighbor walks.
    In the nearest-neighbor case, however, our coupling approach to prove \cref{theorem:monotonicity_recurrence,theorem:monotonicity_transience} is simple and addresses both cases simultaneously.
    
    We hope our scheme of comparing non-homogeneous random walks, where the general case seems entirely intractable, to tractable cases, such as weakly inward-homogeneous walks in \cref{corollary:sandwich_monotonicity}, is a fruitful idea for further work.
    Our paper establishes this comparison via a coupling approach, but these tractable cases, e.g., maximally homogeneous walks, are where Foster--Lyapunov methods are also successful.
    We welcome further work on this topic and leave the application of different methods to our comparison scheme as an open problem for an interested reader.
\end{remark}
\begin{remark}\label{remark:return_tail_bounds}
    We believe that with additional work, one may be able to extract more from our coupling idea in the positive recurrent case, i.e., the proof of \cref{theorem:monotonicity_recurrence} in \cref{section:homogeneous}, namely obtaining comparisons of tail bounds for return times for the two walks, e.g., inequalities of the form $\P(\tau\geq n)\leq C\P(\tau'\geq an+b)$ for constants $C$, $a$, and $b$.
    We leave this as an open problem.
    Such tail bounds help compute moments of passage times, which are important for numerous reasons, including mixing rates, tails of stationary measures, and more.
    In the case of the quadrant, bounds on moments of passage times in the maximally homogeneous case are given by Aspandiiarov, Iasnogorodski, and Menshikov~\cite{AIM}, and we refer readers to \cite{dCMW} for a more recent account.
    However, these methods are unable to directly compare the tail probabilities of two such walks, and so establishing such comparison inequalities, perhaps through a more detailed analysis of this coupling, would be of interest.

    In the special case of \cref{theorem:monotonicity_recurrence} where $X$ has $u_x=u_q$ and $r_y=r_q$, i.e., $X$ is inward-homogeneous and its redirection rules also forbid the positive $x$-axis redirecting up and similarly the positive $y$-axis redirecting right, so that they can only redirect left or down, respectively, if $\tau_X$ and $\tau_Y$ denote the return times of the origin in $X$ and $Y$, respectively, it is straightforward to modify our proof of \cref{theorem:monotonicity_recurrence} to see that we can couple $X$ and $Y$ such that $Y_t \leq X_t$ for all $t$ and thus $\tau_Y\leq \tau_X$, yielding the comparison inequality $\P(\tau_Y \geq n) \leq \P(\tau_X \geq n)$ for all $n$.
\end{remark}

A similar monotonicity question can be asked for slabs $\{(i,j)\in\Z^2:i\geq0,0\leq j \leq k\}$, which can be viewed as an interpolation between the one-dimensional environment of non-homogeneous random walks on $\N$, for which monotonicity always holds trivially, and the two-dimensional environment that we study.
We obtain analogous results for slabs in \cref{sec:slabs}.

\begin{remark}\label{remark:Rayleigh}
    Rayleigh's monotonicity law, introduced by Rayleigh \cite{Rayleigh} in 1899, may be able to make some statements on this question of when monotonicity of recurrence occurs, but is severely limited by the requirement that both $X$ and $Y$ must be Markov chains that correspond to electric networks (see \cite{DoyleSnell} for a comprehensive overview).
    A Markov chain corresponds to an electric network if there exist constants $c_{xy}\geq0$ for every unordered pair of two distinct states $x$ and $y$, such that letting $c_x = \sum_y c_{xy}$ for all states $x$, assuming only finitely many terms in each sum are positive, we have the transition probability from $x$ to $y$ is $c_{xy}/c_x$ and the transition probability from $y$ to $x$ is $c_{xy}/c_y$.
    When the Markov chain is positive recurrent, this is equivalent to the Markov chain being reversible, where the stationary distribution $\pi$ is given by $\pi_x\propto c_x$.
    This imposes certain conditions on a random walk on $\N^2$: for example, for 4 states $w$, $x$, $y$, and $z$, we have $\frac{c_{wx}}{c_{xy}}\cdot\frac{c_{xy}}{c_{yz}}\cdot\frac{c_{yz}}{c_{wz}}\cdot\frac{c_{wz}}{c_{wx}}=1$.
    Letting $(w,x,y,z)=((0,0),(1,0),(1,1),(0,1))$ in an inward-homogeneous random walk on $\N^2$, this implies $\frac{\ell_x d_q r_y u_o}{u_x \ell_q d_y r_o}=1$.
    Considering two other ``squares'' $((1,0),(2,0),(2,1),(1,1))$ and $((0,1),(1,1),(1,2),(0,2))$ yields $\frac{r_x \ell_q}{\ell_x r_q}=1$ and $\frac{u_q d_y}{u_y d_q}=1$.
    These conditions severely limit the applicability of Rayleigh monotonicity to two random walks $X\preceq Y$.
    The restrictiveness of the setting of Rayleigh monotonicity, i.e., reversibility, can also easily be seen in the case of non-nearest-neighbor walks on $\N$.
    
    However, one can use Rayleigh monotonicity to show monotonicity of recurrence holds for trees, in the following sense.
    First, for two random walks $X$ and $Y$ on the same (undirected) tree rooted at $r$, there are two natural ways to define $X\preceq Y$ analogous to our definition for random walks on $\N^2$: 
    we can either simply require that for all states $v\neq r$, the probability of transitioning from $v$ to its parent in $X$ is greater than or equal to that in $Y$; 
    or, we can additionally require that for each state $v$ and for any of its children $w$, the transition probability of $v$ to $w$ in $X$ is less than or equal to that in $Y$.
    Under the former definition, which does not control the relative balance of the transition probabilities from a state to its children, one can construct a counterexample to monotonicity of recurrence, i.e., an example $X\preceq Y$ such that $X$ is transient yet $Y$ is (positive) recurrent.
    We provide such a counterexample in \cref{theorem:tree_counterexample}.
    However, under the latter definition, one can use Rayleigh monotonicity to show that monotonicity of recurrence---as well as monotonicity of the expected return time to the root for positive recurrent chains---holds in general; we establish this in \cref{theorem:tree_monotonicity}.
\end{remark}

These results on $\N^2$, the slab, and trees show that the subset of parameter space that yields recurrent random walks possesses some geometric properties, in this case the structure of an order ideal with respect to some natural partial order.
Motivated by this perspective, 
we consider the more symmetric setting of homogeneous random walks on finitely generated abelian groups.
For a finitely generated abelian group $G$, we can identify the homogeneous random walks that transition according to some probability distribution over some fixed finite subset $S\subseteq G$ with the probability simplex on $S$.
We then consider the subset $R$ of this simplex that corresponds to recurrent random walks, and ask when it possesses certain geometric properties, namely various topological properties and convexity.
We answer some of these questions in \cref{sec:parameter_space}: in particular, we show that $R$ is closed, and if $S$ is symmetric, show $R$ is path-connected and additionally show $R$ is convex if and only if its effective dimension is at most 2 (see \cref{corollary:recurrence_convex_symmetric} for the precise formulation).
We also show its complement $R^c$ is in some sense typically path-connected but not convex.

\begin{remark}\label{remark:chung_fuchs}
    We comment on some of the similarities between our work and that on homogeneous random walks on $\R^d$, i.e., sums of i.i.d.\ random variables with distribution $\mu$.
    In this setting, the Chung--Fuchs theorem \cite{ChungFuchs} completely classifies when the random walk is recurrent, but questions concerning monotonicity and convexity are nontrivial in general.
    Of course, when $d\geq3$ and $\mu$ is nondegenerate with mild regularity, e.g., finite second moments certainly suffice, the Chung--Fuchs theorem implies all walks are transient, while for $d\in\set{1,2}$, if the $d$th moment exists and $\mu$ has mean 0, then the walk is recurrent, so the situation is clear.
    In general, however, when moments need not exist, Mineka \cite{Mineka_Mixture} showed for $d=1$ that mixtures of recurrent random walks need not be recurrent, even when $\mu$ is symmetric, demonstrating an example of non-convexity.
    And in terms of monotonicity, Shepp \cite{Shepp_symmetricRW}, later generalized by Dharmadhikari and Joag-dev \cite{DJ_symmetricRW}, showed that for unimodal transition laws $F$ and $G$, if $F$ is less peaked than $G$, then recurrence of $F$ implies recurrence of $G$.
\end{remark}

In \cref{section:counterexample}, we prove \cref{thm:ex}, and in \cref{section:homogeneous}, we prove \cref{theorem:monotonicity_recurrence,theorem:monotonicity_transience}.
We prove analogues of \cref{thm:ex,theorem:monotonicity_recurrence,theorem:monotonicity_transience,corollary:sandwich_monotonicity} for slabs, i.e., state spaces of the form $\{(i,j)\in\Z^2:i\geq0,0\leq j \leq k\}$ for some positive integer $k$, in \cref{sec:slabs}.
We prove the results about trees, \cref{theorem:tree_counterexample,theorem:tree_monotonicity}, in \cref{sec:trees}.
Lastly, in \cref{sec:parameter_space} we consider whether the subset $R$ of parameter space that yields recurrent random walks, for homogeneous random walks on finitely generated abelian groups, possesses various geometric and topological properties.

\section{A Counterexample to Monotonicity}\label{section:counterexample}
In this section, we show that recurrence, positive recurrence, and transience are not monotonic properties of random walks on $\N^2$ with respect to the partial order $\preceq$.
To do so, we construct a positive recurrent random walk $Y$ and a transient random walk $X$ such that $X\preceq Y$, proving \cref{thm:ex}.
In fact, in our example, the inequalities defining $X\preceq Y$ will all strictly hold.

\begin{remark}\label{remark:Malyshev}
    In a 1972 Russian paper \cite{MalyshevRussian} with English translation \cite{MalyshevEnglish}, as well as earlier in his 1970 book \cite{MalyshevBook}, Malyshev characterized the recurrent homogeneous random walks on $\N^2$, where in this context homogeneous means the mean drift is the same for all $\{(i,j):i,j\geq1\}$, the mean drift is the same on the positive $x$-axis $\{(i,0):i\geq1\}$, and the mean drift is the same on the positive $y$-axis $\{(0,j):j\geq1\}$.
    His result applies to both our examples, i.e., it can be used to verify that $X$ is positive recurrent and $Y$ is transient.
    He did not comment on the non-monotonicity of his recurrence characterization, nor did he provide an example demonstrating monotonicity fails.
    Our proof of \cref{thm:ex} differs from that of Malyshev, with our examples designed to be amenable to much simpler arguments, and we found it instructive to include this proof as we use a similar proof technique in the case of random walks on slabs, which have more complicated boundaries, in \cref{sec:slabs}.
\end{remark}

See \cref{fig:schematic} for a visual schematic of the constructions for $X$ and $Y$.
For each random walk, the transition probabilities for the states on the positive $x$-axis are the same, as well as for the states on the positive $y$-axis, and for the states not on either axis; in that sense, there are only four types of vertices, and we draw the expectation of their transition for each of the four types.

\begin{figure}[tbp]
    \centering
    \begin{subfigure}{65mm}
        \centering
        \begin{tikzpicture}[scale=0.4,font=\normalsize,baseline,thick]
            \draw[->,>=latex,dotted] (0,0) -- (10,0);
            \draw[->,>=latex,dotted] (0,0) -- (0,10);
            \draw[->,>=latex] (6,0) -- (6+4/12*5,1/2*5);
            \draw[->,>=latex] (0,6) -- (1/20*5,6-89/120*5+5/24*5);
            \draw[->,>=latex] (0,0) -- (9/14*5,5/14*5);
            \draw[->,>=latex] (6,6) -- (6-1/6*5,6);
        \end{tikzpicture}
        \caption{The positive recurrent random walk $X$.}
        \label{subfig:X}
    \end{subfigure}
    \begin{subfigure}{65mm}
        \centering
        \begin{tikzpicture}[scale=0.4,font=\normalsize,baseline,thick]
            \draw[->,>=latex,dotted] (0,0) -- (10,0);
            \draw[->,>=latex,dotted] (0,0) -- (0,10);
            \draw[->,>=latex] (6,0) -- (6+31/100*5,49/100*5);
            \draw[->,>=latex] (0,6) -- (1/25*5,6-19/25*5+5/25*5);
            \draw[->,>=latex] (0,0) -- (9/14*5,5/14*5);
            \draw[->,>=latex] (6,6) -- (6-17/100*5,6-73/100*5);
        \end{tikzpicture}
        \caption{The transient random walk $Y$.}
        \label{subfig:Y}
    \end{subfigure}
    \caption{Schematic diagrams of the constructed positive recurrent $X$ and transient $Y$ used to prove \cref{thm:ex}. Arrows show the expected movement starting from the origin, $x$-axis, $y$-axis, and the rest of the quadrant.}
    \label{fig:schematic}
\end{figure}
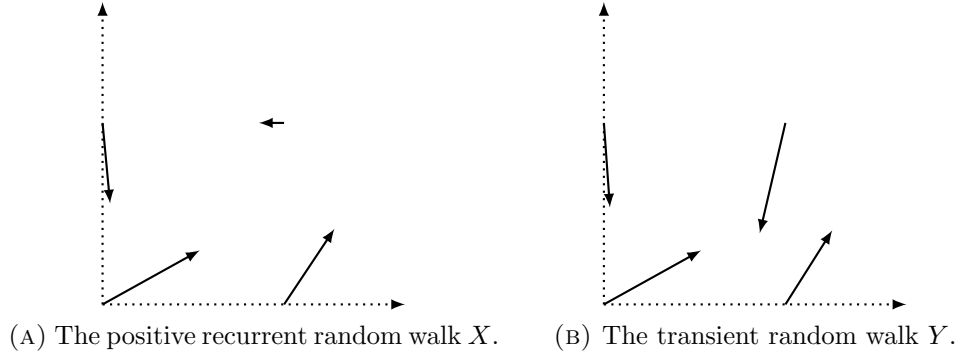

As one can see, the behavior on the $X$-axis is transient while the behavior on the $Y$-axis is recurrent.
For the recurrent walk in the interior of the quadrant, there is a drift to the left while for the transient walk there is a drift downwards.
This is the idea behind the following example.
\subsection{Recurrent Random Walk}\label{subsection:recurrent}
Consider the following random walk on $\N^2$:
\begin{align*}
    \begin{cases} \text{right} & \text{w.p. }\frac{9}{14} \\ \text{up} & \text{w.p. }\frac{5}{14}, \end{cases} \qquad
    \begin{cases} \text{up} & \text{w.p. }\frac{1}{2} \\ \text{right} & \text{w.p. }\frac{5}{12} \\ \text{left} & \text{w.p. }\frac{1}{12}, \end{cases} \qquad
    \begin{cases} \text{down} & \text{w.p. }\frac{89}{120} \\ \text{up} & \text{w.p. }\frac{5}{24} \\ \text{right} & \text{w.p. }\frac{1}{20}, \end{cases} \qquad
    \begin{cases} \text{up} & \text{w.p. }\frac{3}{8} \\ \text{down} & \text{w.p. }\frac{3}{8} \\ \text{left} & \text{w.p. }\frac{5}{24} \\\text{right} & \text{w.p. }\frac{1}{24}, \end{cases}
\end{align*}
for $(0,0)$, $\{(i,0):i\geq1\}$, $\{(0,j):j\geq1\}$, and $\{(i,j):i,j\geq1\}$, respectively.
It is easy to see that this Markov chain is irreducible and 2-periodic.
Hence, to show that it is positive recurrent, it suffices to prove a stationary distribution exists.
It is straightforward to verify that the following distribution is stationary:
\begin{align*}
    \pi_{i,j} =
    \begin{cases}
        \frac{1}{Z}\cdot\left(\frac{5}{7}\right)^{i+j} & i,j\geq 1 \\
        \frac{1}{Z}\cdot\frac{3}{4}\left(\frac{5}{7}\right)^i & j=0\text{ and } i \geq 1 \\
        \frac{1}{Z}\cdot\frac{5}{6}\left(\frac{5}{7}\right)^{j} & i = 0\text{ and } j \geq 1 \\
        \frac{1}{Z}\cdot\frac{35}{72} & (i,j) = (0,0),
    \end{cases}
\end{align*}
where $Z<\infty$ normalizes $\pi$ so that $\sum_{i,j} \pi_{i,j} = 1$.

\subsection{Transient Random Walk}\label{subsection:transient}
Now, consider the following random walk on $\N^2$:
\begin{align*}
    \begin{cases} \text{right} & \text{w.p. }\frac{9}{14} \\ \text{up} & \text{w.p. }\frac{5}{14}, \end{cases} \qquad
    \begin{cases} \text{up} & \text{w.p. }\frac{49}{100} \\ \text{right} & \text{w.p. }\frac{41}{100} \\ \text{left} & \text{w.p. }\frac{1}{10}, \end{cases} \qquad
    \begin{cases} \text{down} & \text{w.p. }\frac{19}{25} \\ \text{up} & \text{w.p. }\frac{5}{25} \\ \text{right} & \text{w.p. }\frac{1}{25}, \end{cases} \qquad
    \begin{cases} \text{up} & \text{w.p. }\frac{1}{100} \\ \text{down} & \text{w.p. }\frac{74}{100} \\ \text{left} & \text{w.p. }\frac{21}{100} \\\text{right} & \text{w.p. }\frac{4}{100}, \end{cases}
\end{align*}
for $(0,0)$, $\{(i,0):i\geq1\}$, $\{(0,j):j\geq1\}$, and $\{(i,j):i,j\geq1\}$, respectively.
Again, this Markov chain is irreducible and 2-periodic.
Note that it strictly satisfies the desired partial order relationship with the Markov chain from \cref{subsection:recurrent} in the sense that for $\{(i,j):i,j\geq1\}$, the left and down probabilities strictly increased while the right and up probabilities strictly decreased, and similarly for the axes, except of course the down probability and left probability for the $x$-axis and $y$-axis, respectively, must be zero and thus do not increase.

To show transience, consider a random walk $X_n$ with $X_0=(0,0)$.
It suffices to show
\[ \sum_{k=1}^\infty \P(X_k=(0,0)) = \sum_{k=1}^\infty \P(X_{2k}=(0,0)) < \infty, \]
as this is the expected number of times $X_n$ returns to $(0,0)$.

Let $S_k = x_{2k} - y_{2k}$ where $X_n=(x_n,y_n)$.
Let $A_k = S_k - S_{k-1}$ for $k\geq1$.
Regardless of $X_{2k-2}$, which is in $\{(i,j):i+j\text{ is even}\}$, we can couple $A_k$ with $B_k$ such that $A_k \geq B_k$, where $B_k$ has marginal distribution
\[ B_k = \begin{cases} -2 & \text{w.p. }\frac{1668}{10000} \\ 0 & \text{w.p. }\frac{6651}{10000} \\ 2 & \text{w.p. }\frac{1681}{10000}, \end{cases} \]
for all $k\geq 1$.
As $A_k\in\{-2,0,2\}$ always, this simply means that regardless of $X_{2k-2}$, we have $A_k=-2$ with probability at most $\frac{1668}{10000}$ and have $A_k=2$ with probability at least $\frac{1681}{10000}$; this can be verified for all possible values of $X_{2k-2}$.
This coupling is tight, i.e., $A_k=B_k$, when $y_{2k-2}=0$ and $x_{2k-2}>0$.
Crucially, the $B_k$ are mutually independent.
Let $T_k = \sum_{\ell=1}^k B_\ell \leq \sum_{\ell=1}^k A_\ell = S_k$. Note that $T_k$ has a positive drift and by Hoeffding's inequality, for all $k\geq1$,
\begin{align*}
    \P(X_{2k}=(0,0))
    &\leq \P(S_k = 0)
    \leq \P(T_k \leq 0)
    \leq \exp\left(-\frac{2\left(\frac{13}{10000}k\right)^2}{16k}\right)
    = \exp\left(\frac{-169k}{80000}\right),
\end{align*}
so
\[ \sum_{k=1}^\infty \P(X_{2k}=(0,0)) \leq \sum_{k=1}^\infty \exp\left(\frac{-169k}{80000}\right) < \infty. \]

\section{Monotonicity in Homogeneous Random Walks}\label{section:homogeneous}
We now consider conditions under which recurrence and transience are indeed monotonic properties with respect to the partial order $\preceq$ and prove \cref{theorem:monotonicity_recurrence,theorem:monotonicity_transience}.
For simplicity of presentation, we prove \cref{theorem:monotonicity_recurrence,theorem:monotonicity_transience} under an inward-homogeneous condition, and then in \cref{remark:weakly_inward_homogeneous_generalization} explain the minor modification needed to extend the proof to the full generality of weak inward-homogeneity.

We will use the following notation. 
Let $\leq$ denote the partial order on $\Z^2$ given by $(i_1,j_1)\leq(i_2,j_2)$ if $i_1\leq i_2$ and $j_1\leq j_2$.
For two elements $x=(i_1,j_1)$ and $y=(i_2,j_2)$ in $\Z^2$, we let $x+y$ denote $(i_1+i_2,j_1+j_2)$.
We view $\N^2$ as a subset of $\Z^2$.
\begin{proof}[Proof of \cref{theorem:monotonicity_recurrence}]
Let $X$ and $Y$ both start at the origin.
Note that irreducibility of $X$ implies it is elliptic, i.e., all probabilities in \cref{definition:inward_homogeneous} are positive.
It suffices to show that we can couple $X$ and $Y$ such that $Y_t \leq X_t+s_t$ for some $s_t\in\{(2,0),(1,1),(0,2)\}$ for all $t$.
Let $S=\{(2,0),(1,1),(0,2)\}$.
This is because for each time $t$ such that $X_t=(0,0)$, we have $Y_t\in\{(0,0),(2,0),(1,1),(0,2)\}$, as $X$ and $Y$ are both 2-periodic, so the probability that $Y_{t+2}=(0,0)$ is at least some positive constant $p>0$.
So if $X$ is recurrent and thus $X_t=(0,0)$ for infinitely many $t$ regardless of $X_0$, the probability that $Y$ never returns to $(0,0)$ is 0.
And if $X$ is positive recurrent so that some upper bound $\tau<\infty$ is greater than or equal to the four expected hitting times of $(0,0)$ starting at one of the four states in $S\cup\{(0,0)\}$, then the expected return time for $Y$ is at most $\frac{\tau+2}{p}$, and thus $Y$ is also positive recurrent.

As $Y\preceq X$, we can view each step of $Y$ as following a two-stage process: at the current state, select a direction based on the transition probabilities of $X$ at this state, and then if the direction is up or right, with some probability we instead go left or down.
The second ``inward shift'' step converts the transition probabilities of $X$ at this state to the desired transition probabilities of $Y$ at this state.

Using the redirection framework for inward-homogeneity, we can subdivide each step of $X$ with another two-stage process: first selecting left, right, up, or down based on $\ell_q$, $r_q$, $u_q$, and $d_q$, and then redirecting as necessary.
This subdivision transforms the process of determining each step of $Y$ into a three-stage process.

We can now couple $X$ and $Y$ as follows, so that the first two stages of $Y$ coincide with the two stages of $X$.
We first use the same selection of left, right, up, or down based on $\ell_q$, $r_q$, $u_q$, and $d_q$.
Then $X$ and $Y$ both redirect as necessary, but if $X$ and $Y$ use the same redirection rule, i.e., both are on the positive $x$-axis, both are on the positive $y$-axis, or both are at the origin, the same redirection choice is made.
Finally, if the selected move for $Y$ is up or right, we inward shift as necessary.

It now remains to verify that this coupling satisfies $Y_t \leq X_t + s_t$ for some $s_t\in S$ for all $t$, which we prove by induction.
We start with $X_0=Y_0=(0,0)$, and for fixed $t\geq 1$, suppose that $Y_{t-1} \leq X_{t-1} + s_{t-1}$ for some $s_{t-1}\in S$.
Let $Y_t' \geq Y_t$ be the state obtained from $Y_t$ by ignoring inward shifting at timestep $t$; it suffices to show $Y_t' \leq X_t + s_t$ for some $s_t\in S$.
The only way $Y_t' \leq X_t + s_{t-1}$ could not hold is if $X$ and $Y$ have different redirection rules (with no redirection being a fourth possibility), as otherwise $Y_t'-Y_{t-1}=X_t-X_{t-1}$ and thus $Y_t' \leq X_t + s_{t-1}$ then follows from the inductive hypothesis.

We split into cases depending on the selection of left, right, up, or down based on $\ell_q$, $r_q$, $u_q$, and $d_q$, i.e., the first stage of the three-stage coupling.
Note that up and right movements never get redirected, so we only need to check down and left.
Let $X_t=(i_t^{(1)},j_t^{(1)})$ and let $Y_t=(i_t^{(2)},j_t^{(2)})$.
Also let $Y_t'=(i_t^{(2')},j_t^{(2')})$.
Note that the 2-periodicity of our random walks implies
\begin{equation}\label{equation:parity}
    i_t^{(1)}+j_t^{(1)}\equiv i_t^{(2)}+j_t^{(2)}\equiv i_t^{(2')}+j_t^{(2')}\equiv t\pmod{2}.
\end{equation}

\noindent\textbf{Case 1}: left.
Redirections only occur on the $y$-axis.
If both $X$ and $Y'$ redirect but under different rules, then one of them is at the origin while the other is on the positive $y$-axis.
If $Y_{t-1}=(0,0)$, then $Y_t' \leq (1,1) \leq X_t + (1,1)$.
Otherwise, $X_{t-1}=(0,0)$ and $Y_{t-1}=(0,2)$, where $Y'$ can only move down or right; if $Y'$ moves right, then $Y_t'-X_t=(1,1)$ or $(0,2)$ if $X$ moves up or right, respectively, and if $Y'$ moves down, then $Y_t' \leq X_t + (1,1)$.

If instead exactly one of $X$ and $Y'$ redirects, first suppose $X$ redirects while $Y'$ moves left, so that $i_{t-1}^{(1)}=0$ and $i_{t-1}^{(2)}\in\{1,2\}$.
If $i_{t-1}^{(2)}=2$ so that $i_t^{(2')}=1$, we have $s_{t-1}=(2,0)$ so $j_{t-1}^{(2)}\leq j_{t-1}^{(1)}$, and thus $Y_t' \leq X_t + (1,1)$.
Otherwise $i_{t-1}^{(2)}=1$ so that $i_t^{(2')}=0$, and thus $s_{t-1}\in\{(1,1),(2,0)\}$ so $j_{t-1}^{(2)} \leq j_{t-1}^{(1)}+1$, and thus $Y_t' \leq X_t + (0,2)$.

Lastly, suppose $Y'$ redirects while $X$ moves left, so that $i_{t-1}^{(2)}=0$ and $i_{t-1}^{(1)} \geq 1$.
If $Y'$ moves up, meaning $Y_{t-1}=(0,0)$, then $Y_t' \leq X_t + (0,2)$.
As we always have $j_{t-1}^{(2)} \leq j_{t-1}^{(1)}+2$, if $Y'$ moves down, then $Y_t' \leq X_t + (0,2)$.
If $Y'$ moves right, then $Y_t' \leq X_t + (0,2)$ if $i_{t-1}^{(1)}\geq 2$, and if $i_{t-1}^{(1)}=1$, by \eqref{equation:parity} we have $j_{t-1}^{(2)} \leq j_{t-1}^{(1)}+1$, so $Y_t' \leq X_t + (1,1)$.

\noindent\textbf{Case 2}: down.
The argument is identical to case 1, but reflected across the line $y=x$.
\end{proof}
\begin{remark}\label{remark:monotonicity_recurrence_irreducible}
    The assumptions of irreducibility in \cref{theorem:monotonicity_recurrence} can be removed, rephrasing recurrence as recurrence of the state $(0,0)$.
    The proof of \cref{theorem:monotonicity_recurrence} only uses irreducibility to imply ellipticity, which is used to show that if $Y_t\in\{(0,0),(2,0),(1,1),(0,2)\}$, then in each of these four cases we have $Y_{t+2}=(0,0)$ with positive probability.
    However, this follows from the assumptions that $Y\preceq X$ and $X$ is inward-homogeneous.
    If $Y_t=(0,0)$ implies $Y_{t+2}\neq(0,0)$ with probability 1, without loss of generality suppose $Y_{t+1}=(1,0)$ with positive probability.
    Then $Y$ cannot move left from $(1,0)$, implying $\ell_q=\ell_x=0$ for $X$, so with some probability $X_1=(1,0)$, from which it cannot return to $(0,0)$, contradicting recurrence of $X$.
    Essentially identical arguments address the cases $Y_t=(2,0)$ and $Y_t=(0,2)$.
    For $Y_t=(1,1)$, there are four cases for how one could have $Y_{t+2}\neq(0,0)$ with probability 1; we show all contradict recurrence of $(0,0)$ in $X$.

    \noindent\textbf{Case 1}: $Y$ cannot move down from $(0,1)$ and cannot move left from $(1,0)$.
    Then $X$ also cannot make these moves, and thus cannot ever return to $(0,0)$.

    \noindent\textbf{Case 2}: $Y$ cannot move down from $(0,1)$ and cannot move down from $(1,1)$.
    Then $X$ cannot ever move down, yet as $Y_t=(1,1)$, with positive probability $X$ can reach $(1,1)$, from which it cannot return to $(0,0)$.

    \noindent\textbf{Case 3}: $Y$ cannot move left from $(1,1)$ and cannot move left from $(1,0)$.
    This case is the reflection of case 2, and essentially the same argument holds.

    \noindent\textbf{Case 4}: $Y$ cannot move left from $(1,1)$ and cannot move down from $(1,1)$.
    Then $X$ cannot move left or down in $\{(i,j):i,j\geq1\}$, yet as $Y_t=(1,1)$, with positive probability $X$ can reach $\{(i,j):i,j\geq1\}$, from which it cannot return to $(0,0)$.
\end{remark}

We now prove \cref{theorem:monotonicity_transience}. 
\begin{proof}[Proof of \cref{theorem:monotonicity_transience}]
    Let $X$ and $Y$ both start at $(0,0)$.
    Let $S=\{(2,0),(1,1),(0,2)\}$.
    It suffices to show that we can couple $X$ and $Y$ such that $Y_t \geq X_t - s_t$ for some $s_t\in\{(2,0),(1,1),(0,2)\}$ for all $t$.
    This is because $X$ being transient implies that with some positive probability, $X$ visits $\{(i,j):i+j\leq 2\}$ finitely often, and thus $Y$ returns to $(0,0)$ finitely often.

    We use the same coupling as in the proof of \cref{theorem:monotonicity_recurrence}, except in the third stage, if the selected move for $Y$ is down or left, we may apply an ``outward shift'' and instead go up or right.
    We prove $Y_t \geq X_t - s_t$ for some $s_t\in S$ for all $t$ by induction on $t$.
    We start with $X_0=Y_0=(0,0)$, and for fixed $t\geq 1$, suppose $Y_{t-1} \geq X_{t-1}-s_{t-1}$ for some $s_{t-1}\in S$.
    Let $Y_t'$ be the state obtained from $Y_t$ by ignoring outward shifting at timestep $t$; it suffices to show $Y_t' \geq X_t - s_t$ for some $s_t\in S$, as $Y_t \geq Y_t'$ always.
    Similar to before, we only need to consider when $X$ and $Y$ have different redirection rules, and we split into cases depending on the selection of left, right, up, or down.
    As up and right movements never get redirected, by reflecting across the line $y=x$, without loss of generality it suffices to address the left case.
    Let $X_t=(i_t^{(1)}, j_t^{(1)})$, $Y_t=(i_t^{(2)},j_t^{(2)})$, and $Y_t'=(i_t^{(2')},j_t^{(2')})$.

    As redirections only occur on the $y$-axis, if both $X$ and $Y'$ redirect but under different rules, then one of them is at the origin while the other is on the positive $y$-axis.
    If $X_{t-1}=(0,0)$, then $Y_t' \geq (1,1) - (1,1) \geq X_t - (1,1)$.
    Otherwise $Y_{t-1}=(0,0)$, so $X_{t-1}=(0,2)$, where $X$ can only move down or right; if $X$ moves right, then $X_t-Y_t'=(1,1)$ or $X_t-Y_t'=(0,2)$ if $Y'$ moves up or right, respectively, and if $X$ moves down, then $Y_t' \geq X_t-(1,1)$.

    If instead exactly one of $X$ and $Y'$ redirects, first suppose $Y$ redirects while $X'$ moves left, so that $i_{t-1}^{(2)}=0$ and $i_{t-1}^{(1)}\in\{1,2\}$.
    If $i_{t-1}^{(1)}=2$ so that $i_t^{(1)}=1$, we have $s_{t-1}=(2,0)$ so $j_{t-1}^{(2)} \geq j_{t-1}^{(1)}$, and thus $Y_t' \geq X_t - (1,1)$.
    Otherwise $i_{t-1}^{(1)}=1$ so that $i_t^{(1)}=0$, and thus $s_{t-1}\in\{(1,1),(2,0)\}$ so $j_{t-1}^{(2)} \geq j_{t-1}^{(1)} - 1$, and thus $Y_t' \geq X_t - (0,2)$.

    Lastly, suppose $X$ redirects while $Y'$ moves left, so that $i_{t-1}^{(1)}=0$ and $i_{t-1}^{(2)} \geq 1$.
    If $X$ moves up, meaning $X_{t-1}=(0,0)$, then $Y_t' \geq (0,-1) = X_t - (0,2)$.
    As we always have $j_{t-1}^{(2)} \geq j_{t-1}^{(1)} - 2$, if $X$ moves down, then $Y_t' \geq X_t - (0,2)$.
    If $X$ moves right, then $Y_t' \geq X_t - (0,2)$ if $i_{t-1}^{(2)} \geq 2$, and if $i_{t-1}^{(2)}=1$, by \eqref{equation:parity} we have $j_{t-1}^{(2)} \geq j_{t-1}^{(1)} - 1$, so $Y_t' \geq X_t - (1,1)$.
\end{proof}
\begin{remark}\label{remark:weakly_inward_homogeneous_generalization}
    Now assume $X$ is only weakly inward-homogeneous.
    Then when $X$ is on the positive $x$-axis, we can think of it as being lazy with probability $p_x:=\frac{r_x-r_q}{r_x}=1-\frac{r_q}{r_x}$, go right with probability $\frac{r_q}{r_x}r_x=r_q$, go left with probability $\frac{r_q}{r_x}\ell_x\geq\ell_q$, and go up with probability $\frac{r_q}{r_x}u_x\geq u_q$, so that its movement probabilities once it moves is still $\ell_x$, $u_x$, and $r_x$, as required.
    Similarly let $X$ be lazy on the positive $y$-axis.
    This laziness does not change whether $X$ is (positive) recurrent or transient.
    Then couple $(X,Y)$ as before with our three-stage process of picking one of the four directions, redirecting, and inward shifting, except now if say we need to redirect from choosing to go down while on the positive $x$-axis, with probability $\frac{p_x}{d_q}$ we do not move, and otherwise we redirect as necessary; of course, if $X$ and $Y$ are both on the positive $x$-axis, the decision to stay put or redirect is shared between $X$ and $Y$.
    Then it's easy to verify that the inequalities $Y_t \leq X_t + s_t$ and $Y_t\geq X_t - s_t$ in the proofs of \cref{theorem:monotonicity_recurrence,theorem:monotonicity_transience}, respectively, still hold with this coupling, and so the proof follows identically.
\end{remark}

\section{Random Walks on a Slab} \label{sec:slabs}

For positive integer $k$, let the \textit{slab} of \textit{thickness} $k$ be $\{(i,j)\in\Z^2:i\geq0,0\leq j \leq k\}$.
Intuitively, the slab is an intermediate environment for random walks between the 1-dimensional random walks on $\N$ and the 2-dimensional random walks on $\N^2$, which are recovered in the cases $k=0$ and $k=\infty$, respectively.
For a slab of thickness $k$, we partition the slab into six (possibly empty) regions: the \textit{center} $\{(i,j):i\geq1,0<j<k\}$, the \textit{upper boundary} $\{(i,k):i\geq1\}$, the \textit{lower boundary} $\{(i,0):i\geq1\}$, the \textit{left boundary} $\{(0,j):0<j<k\}$, the origin $(0,0)$, and the \textit{upper corner} $(0,k)$.

In \cref{subsection:slab_counterexample}, similar to \cref{section:counterexample}, we provide a counterexample to monotonicity of recurrence, positive recurrence, and transience of random walks on the slab with respect to the partial order $\preceq$ defined analogously to \cref{definition:N2_partial_order}.
The construction works for all $k\geq2$, and while a different construction can work for $k=1$, for sake of brevity we do not study the case $k=1$.
Furthermore, in \cref{subsection:slab_homogeneity}, similar to \cref{section:homogeneous}, we show that for $k\geq2$, monotonicity holds under a different and arguably more natural partial order of random walks on the slab, assuming a homogeneity condition similar to \cref{definition:inward_homogeneous} of inward-homogeneity.
Recalling the definition of inward-homogeneity is based on transition probabilities $\ell_q$, $r_q$, $u_q$, and $d_q$ in the general case of states lying on neither axis, this concept of homogeneity does not neatly transfer to the case $k=1$, where there is no center.
Thus, we henceforth restrict attention to $k\geq2$.

\subsection{A counterexample to monotonicity}\label{subsection:slab_counterexample}
Using the essentially identical definition of $\preceq$ for random walks on $\N^2$ from \cref{definition:N2_partial_order} for random walks on the slab, we first show that recurrence, positive recurrence, and transience are not monotonic properties of random walks on the slab with respect to the partial order $\preceq$.
In this sense, random walks on the slab are similar to random walks on $\N^2$, and our construction of a positive recurrent random walk $X$ and a transient random walk $Y$ such that $Y\preceq X$ follows the same intuition.
For our random walk on $\N^2$, the positive $x$-axis moves outwards while the positive $y$-axis moves inwards and the primary change from $X$ to $Y$ is changing the center from predominantly moving towards the $y$-axis to predominantly moving towards the $x$-axis.
For our slab construction, we preserve the role of the positive $x$-axis, i.e., the lower boundary, and use the upper boundary in place of the positive $y$-axis.

However, our proof technique for the recurrent side is different.
Fix positive integer $k\geq2$, and consider the following random walk on the slab of thickness $k$:
\begin{align*}
    \begin{cases} \text{up} & \text{w.p. }0.65 \\ \text{right} & \text{w.p. }0.33 \\ \text{left} & \text{w.p. }0.01 \\ \text{down} & \text{w.p. }0.01, \end{cases} \qquad\qquad
    \begin{cases} \text{right} & \text{w.p. }0.52 \\ \text{up} & \text{w.p. }0.47 \\ \text{left} & \text{w.p. }0.01, \end{cases} \qquad\qquad
    \begin{cases} \text{down} & \text{w.p. }0.49 \\ \text{left} & \text{w.p. }0.48 \\ \text{right} & \text{w.p. }0.03, \end{cases} \\
    \begin{cases} \text{up} & \text{w.p. }0.97 \\ \text{right} & \text{w.p. }0.02 \\ \text{down} & \text{w.p. }0.01, \end{cases} \qquad\qquad
    \begin{cases} \text{right} & \text{w.p. }0.99 \\ \text{up} & \text{w.p. }0.01, \end{cases} \qquad\qquad
    \begin{cases} \text{down} & \text{w.p. }0.50 \\ \text{right} & \text{w.p. }0.50, \end{cases}
\end{align*}
for the center, lower boundary, upper boundary, left boundary, origin, and upper corner, respectively.
It is easy to see that this Markov chain is irreducible and 2-periodic.
To be explicit about the 2-periodicity, we partition the slab of thickness $k$ into the two sets $T_0=\{(i,j)\in\Z^2:i\geq0,0\leq j \leq k, i+j\text{ is even}\}$ and $T_1=\{(i,j)\in\Z^2:i\geq0,0\leq j \leq k, i+j\text{ is odd}\}$, where starting in a state in $T_0$ we must move to a state in $T_1$, and vice versa.
The precise argument to show positive recurrence depends on the parity of $k$.
We first assume $k$ is even, where the argument is slightly simpler.

We claim that it suffices to show that for any starting state in $T_0$ within $k$ distance of at least one of the two corners, i.e., the origin and upper corner, the expected hitting time of $\{(0,0),(0,k)\}$, meaning the expected time to reach either of the two corners, not including time 0 if starting at one of the two corners, is finite.
If this holds, then let $r<\infty$ be greater than or equal to all of these expected hitting times.
There exists some probability $p>0$ such that with at least probability $p$, we reach the upper corner after exactly $k$ steps, starting at either of the two corners.
It suffices to show that the expected return time of the upper corner is finite, but we can bound our expected return time by $\frac{r+k}{p}<\infty$, by considering each time we return to one of the two corners, which takes at most $r$ time in expectation, and then whether or not we reach the upper corner after $k$ more steps, which occurs with probability $p$; if not, we are at one of the starting states in the definition of $r$, and we try again, where each trial takes at most $r+k$ time in expectation and has success probability at least $p$.

Now, consider starting at a state in $T_0$ within $k$ distance of at least one of the two corners, and let our random walk be $X_t$.
Let $\tau$ be the hitting time of $\{(0,0),(0,k)\}$, so that we must show $\E[\tau]<\infty$.
However, modify $X_t$ such that if we return to either of the two corners, we stop, i.e., for all $t\geq\tau$, we have $X_t=X_\tau$.
Similar to the proof of transience in \cref{subsection:transient}, let $S_n=x_{2n}-y_{2n}$ where $X_{t}=(x_{t},y_{t})$.
As $X_{2n}\in T_0$ for all nonnegative integers $n$, it is straightforward, though slightly tedious, to verify that $\E[S_{n+1}-S_n|S_n]\leq-0.0892$ regardless of the value of $S_n\in T_0$, for $n\geq 1$.
This inequality is tight when $S_n$ is on the lower boundary.
Crucially, without the modification of $X_t$ such that we stop upon returning to either of the two corners, this inequality would be violated if $S_n\in\{(0,0),(0,k)\}$.
We have $X_0=(x_0,y_0)$ is fixed, with $S_n\geq -k$ for all $n$, so
\[ -k-x_0+y_0 \leq \lim_{n\to\infty} \E[S_n - S_0]. \]
As $\abs{S_1-S_0}\leq2$ always, we have 
\begin{align*}
    -k-x_0+y_0-2 &\leq \lim_{n\to\infty} \E[S_n - S_1] = \sum_{n=1}^\infty \E[S_{n+1}-S_n]
    \\ &\leq \sum_{n=1}^\infty -0.0892 \cdot \Pr(X_{2n}\not\in\{(0,0),(0,k)\})
    \\ &=-0.0892 \sum_{n=1}^\infty \Pr(\tau > 2n)
    = 0.0892 - 0.0446 \cdot \E[\tau],
\end{align*}
where the last equality holds because $\tau$ must be a positive even integer.
This bounds
\[ \E[\tau] \leq \frac{k+x_0-y_0+2}{0.0446}+2 < \infty. \]

The argument when $k$ is odd is nearly identical, with some adjustments due to the fact that $(0,k)$ lies in $T_1$ instead of $T_0$.
Let $U=\{(0,k),(1,0),(0,1)\}$.
Applying essentially the same argument as when $k$ is even, it suffices to show that for any starting state in $T_1$ within $k+1$ distance of at least one of the points in $U$, the expected hitting time of $U$ is finite.
If we let $r<\infty$ be an upper bound on these expected hitting times and let $p$ be a lower bound on the probability of reaching $(0,k)$ after exactly $k+1$ steps, starting at one of the three points in $U$, then we can bound our expected return time of the upper corner by $\frac{r+k+1}{p}<\infty$.

Now, consider starting at a state in $T_1$ within $k+1$ distance of at least one of the points in $U$, and let this random walk be $X_t$, modified so that if we return to $U$, we stop.
Let $\tau$ be the hitting of $U$.
Again letting $S_n=x_{2n}-y_{2n}$, we continue to have $\E[S_{n+1}-S_n|S_n]\leq-0.0892$ regardless of the value of $S_n\in T_1$, for $n\geq 1$.
Again, the modification to stop upon reaching $U$ is critical for this inequality to work, as the inequality would be violated if we moved from either of the corners, which cannot occur for $n\geq 1$ by the stopping rule, though it can occur for $n=0$ if we started in $U$.
The same argument as before yields $\E[\tau]\leq\frac{k+x_0-y_0+2}{0.0446}+2<\infty$, which completes the proof of positive recurrence.

Our proof technique for the transient random walk on the slab is identical to that on $\N^2$ in \cref{subsection:transient}.
Consider the following random walk on the slab of thickness $k$:
\begin{align*}
    \begin{cases} \text{up} & \text{w.p. }0.01 \\ \text{right} & \text{w.p. }0.32 \\ \text{left} & \text{w.p. }0.02 \\ \text{down} & \text{w.p. }0.65, \end{cases} \qquad\qquad
    \begin{cases} \text{right} & \text{w.p. }0.51 \\ \text{up} & \text{w.p. }0.46 \\ \text{left} & \text{w.p. }0.03, \end{cases} \qquad\qquad
    \begin{cases} \text{down} & \text{w.p. }0.50 \\ \text{left} & \text{w.p. }0.49 \\ \text{right} & \text{w.p. }0.01, \end{cases} \\
    \begin{cases} \text{up} & \text{w.p. }0.01 \\ \text{right} & \text{w.p. }0.01 \\ \text{down} & \text{w.p. }0.98, \end{cases} \qquad\qquad
    \begin{cases} \text{right} & \text{w.p. }0.99 \\ \text{up} & \text{w.p. }0.01, \end{cases} \qquad\qquad
    \begin{cases} \text{down} & \text{w.p. }0.51 \\ \text{right} & \text{w.p. }0.49, \end{cases}
\end{align*}
for the center, lower boundary, upper boundary, left boundary, origin, and upper corner, respectively.
Again, this Markov chain is irreducible and 2-periodic, and strictly satisfies the desired partial order relationship with the positive recurrent Markov chain in the sense that the left and down probabilities strictly increased while the right and up probabilities strictly decreased, whenever they are allowed to be positive.
We use the same notation of $S_n$, $A_n$, and $B_n$ from \cref{subsection:transient}, where $B_n$ has marginal distribution
\[ B_n = \begin{cases} -2 & \text{w.p. }0.2401 \\ 0 & \text{w.p. }0.4998 \\ 2 & \text{w.p. }0.2601. \end{cases} \]
The $-2$ probability is tight when on the upper boundary, and the $2$ probability is tight when on the lower boundary.
As $\E[B_k]>0$, our Hoeffding argument from \cref{subsection:transient} implies this random walk is transient.

\subsection{Monotonicity in Homogeneous Random Walks}\label{subsection:slab_homogeneity}
As the slab has finite thickness, there is no reason to believe that increasing the down probabilities will increase the chance that the random walk is recurrent; by reflecting across $y=k/2$, the up and down probabilities play identical roles.
And as seen in \cref{subsection:slab_counterexample}, increasing the down probabilities can allow for counterexamples of monotonicity.
To this end, we introduce a more natural definition for a partial order on random walks on the slab.
\begin{definition}\label{definition:slab_partial_order}
    For two random walks $X$ and $Y$ on the slab of thickness $k$, say $X\trianglelefteq Y$, or equivalently $Y\trianglerighteq X$, if for each state, the probability of going left is weakly greater in $X$ than in $Y$, the probability of going right is weakly greater in $Y$ than in $X$, the probability of going up is the same in $X$ as in $Y$, and the probability of going down is the same in $X$ as in $Y$.
\end{definition}
Intuitively, $X\trianglelefteq Y$ means $X$ goes left more than $Y$.
Similar to \cref{definition:inward_homogeneous}, we introduce a definition of homogeneity on the slab such that monotonicity holds.
\begin{definition}\label{definition:slab_homogeneous}
    Say a random walk on the slab is \textit{homogeneous} if it is of the following form:
    {\allowdisplaybreaks
    \begin{align*}
        \begin{cases} \text{left} & \text{w.p. }\ell_q \\ \text{right} & \text{w.p. }r_q \\ \text{up} & \text{w.p. }u_q \\ \text{down} & \text{w.p. }d_q, \end{cases} \qquad\qquad
        \begin{cases} \text{left} & \text{w.p. }\ell_x \\ \text{right} & \text{w.p. }r_x \\ \text{up} & \text{w.p. }u_x, \end{cases} \qquad\qquad
        \begin{cases} \text{left} & \text{w.p. }\ell_u \\ \text{right} & \text{w.p. }r_u \\ \text{down} & \text{w.p. }d_u, \end{cases} \\
        \begin{cases} \text{down} & \text{w.p. }d_y \\ \text{up} & \text{w.p. }u_y \\ \text{right} & \text{w.p. }r_y, \end{cases} \qquad\qquad
        \begin{cases} \text{right} & \text{w.p. }r_o \\ \text{up} & \text{w.p. }u_o, \end{cases} \qquad\qquad
        \begin{cases} \text{right} & \text{w.p. }r_c \\ \text{down} & \text{w.p. }d_c, \end{cases}
    \end{align*}}
for the center, lower boundary, upper boundary, left boundary, origin, and upper corner, respectively, where
\begin{align*}
    r_c \geq r_u \geq r_q \leq r_x \leq r_o, \quad r_y \geq r_q, \quad \ell_u \geq \ell_q \leq \ell_x, \quad d_u, d_y, d_c \geq d_q, \quad\text{and}\quad u_x, u_y, u_o \geq u_q.
\end{align*}
\end{definition}
Similar to an inward-homogeneous random walk on $\N^2$, we can view a homogeneous random walk on the slab as being given by the four transition probabilities from the general center case, along with redirection rules if a forbidden move is selected.
Within this redirection framework, the only nontrivial inequalities are $r_c \geq r_u$ and $r_o \geq r_x$, which can be interpreted as the existence of a coupling of the upper corner and boundary redirection rules such that if both attempt to move up and the upper boundary redirects to the right, then the upper corner must redirect to the right as well, and symmetrically for the origin and lower boundary.

We now show that under a homogeneity condition, recurrence and positive recurrence are monotonic properties with respect to $\trianglelefteq$.
\begin{theorem}\label{theorem:slab_recurrence}
    For positive integer $k\geq2$, if an irreducible homogeneous random walk $X$ on the slab of thickness $k$ is recurrent (respectively, positive recurrent), then any irreducible random walk $Y$ on the slab of thickness $k$ such that $Y\trianglelefteq X$ is also recurrent (respectively, positive recurrent).
\end{theorem}
\begin{proof}
    Our proof is structurally similar to that of \cref{theorem:monotonicity_recurrence}, so for sake of brevity and clarity we will omit additional elaboration when the reasoning is essentially identical to before.
    It suffices to show that we can couple $X$ and $Y$ such that $\abs{j_t^{(2)}-j_t^{(1)}}+\max\{i_t^{(2)}-i_t^{(1)},0\}\leq 2\ceil{k/2}$ for all $t$.
    We couple $X$ and $Y$ in the same manner as before, where inward shifting only ever changes a right movement in $Y$ to a left movement.
    We couple the upper corner and boundary redirection rules as described in \cref{definition:slab_homogeneous}, which implies that if one of $X$ and $Y$ is on the upper boundary and the other is at the upper corner, then ignoring inward shifting, if they (both) attempt to move up and the one on the upper boundary redirects right, the one at the upper corner will redirect right as well.
    We similarly couple the origin and lower boundary redirection rules.

    We use the same notation of $Y_t'=(i_t^{(2')},j_t^{(2')})$ as before.
    It suffices to show that for fixed $t\geq1$, if $\abs{j_{t-1}^{(2)}-j_{t-1}^{(1)}}+\max\{i_{t-1}^{(2)}-i_{t-1}^{(1)},0\}\leq2\ceil{k/2}$, then $\abs{j_t^{(2')}-j_t^{(1)}}+\max\{i_t^{(2')}-i_t^{(1)},0\}\leq2\ceil{k/2}$, as inward shifting weakly decreases the second term.
    If neither $X$ nor $Y$ redirect, the desired inequality trivially holds by induction.
    Now, suppose there is exactly one redirection; we divide into casework on who redirects and the direction that the random walks attempted to go in.

    \noindent\textbf{Case 1}: $X$ redirects trying to go up.
    Then $Y$ goes up, decreasing the first term by 1 from time $t-1$ to $t$, so the inequality holds regardless of how $X$ redirects, as its movement can only change one term by at most 1.

    \noindent\textbf{Case 2}: $X$ redirects trying to go down.
    This case is symmetrically equivalent to case 1.

    \noindent\textbf{Case 3}: $X$ redirects trying to go left.
    Then $Y$ goes left, decreasing the second term by 1, so the inequality holds regardless of how $X$ redirects.

    \noindent\textbf{Case 4}: $Y$ redirects trying to go up.
    This case follows identically to case 1.

    \noindent\textbf{Case 5}: $Y$ redirects trying to go down.
    This case is symmetrically equivalent to case 4.

    \noindent\textbf{Case 6}: $Y$ redirects trying to go left.
    This implies $i_t^{(2')}-i_t^{(1)}\leq0$, so we trivially bound $\abs{j_t^{(2')}-j_t^{(1)}}+\max\{i_t^{(2')}-i_t^{(1)},0\}\leq k + 0 \leq 2\ceil{k/2}$.

    Finally, suppose there are two different redirections, meaning $X$ and $Y$ are in two different non-central regions; we divide into casework on the direction that the random walks attempted to go in.
    Note that we can ignore any cases where the redirection rules pick the same direction, as then the two terms are unchanged.

    \noindent\textbf{Case 1}: $X$ and $Y$ redirect trying to go up.
    If $X$ is to the right, by 2-periodicity we bound the first term by 1 and the second term by 0.
    If $Y$ is to the right, meaning $X$ is at the upper corner, the situation is more complicated.
    If $Y$ redirects down, then (ignoring $X$ moving down as then their movements will be the same) $X$ moves right, so the first term increases by 1 while the second term decreases by 1, so the inequality holds by induction.
    If $Y$ redirects left then the second term decreases by 1, so the inequality holds regardless of how $X$ redirects.
    If $Y$ redirects right, our coupling implies $X$ redirects right, which implies the inequality continues to hold.

    \noindent\textbf{Case 2}: $X$ and $Y$ redirect trying to go down.
    This case follows symmetrically to case 1.

    \noindent\textbf{Case 3}: $X$ and $Y$ redirect trying to go left.
    As $X$ and $Y$ are both on the $y$-axis, by 2-periodicity they are an even distance apart.
    We have $\max\{i_{t-1}^{(2)}-i_{t-1}^{(1)},0\}=0$.
    If $k$ is odd, or if $k$ is even and they are at most $k-2$ apart, then $\abs{j_{t-1}^{(2)}-j_{t-1}^{(1)}}\leq 2\ceil{k/2}-2$ and the sum of the two terms can increase by at most 2, yielding the desired inequality.
    Otherwise, $X$ and $Y$ are on the two corners.
    If either moves vertically then the first term decreases by 1 so the inequality holds regardless of how the other redirects; otherwise, they both move right, in which case the inequality continues to hold.
\end{proof}
We also show that transience is a monotonic property with respect to $\trianglelefteq$, assuming homogeneity on one side.
\begin{theorem}\label{theorem:slab_transience}
    For positive integer $k\geq2$, if an irreducible homogeneous random walk $X$ on the slab of thickness $k$ is transient, then any irreducible random walk $Y$ on the slab of thickness $k$ such that $Y\trianglerighteq X$ is also transient.
\end{theorem}
\begin{proof}
    Following the reasoning of the proof of \cref{theorem:monotonicity_transience} and using the notation of the proof of \cref{theorem:slab_recurrence}, it suffices to couple $X$ and $Y$ such that $\abs{j_{t}^{(1)}-j_{t}^{(2)}}+\max\{i_t^{(1)}-i_t^{(2)},0\}\leq 2\ceil{k/2}$ for all $t$.
    The coupling used is the same as that in the proof of \cref{theorem:slab_recurrence}, where inward shifting is replaced by outward shifting as in the proof of \cref{theorem:monotonicity_transience}.
    The inductive verification that the coupling satisfies the desired inequality is identical to that in the proof of \cref{theorem:slab_recurrence}, where in all the casework we swap the roles of $X$ and $Y$. 
\end{proof}
Analogous to \cref{corollary:sandwich_monotonicity}, the following corollary of \cref{theorem:slab_recurrence,theorem:slab_transience} allows for application of monotonicity when neither random walk is homogeneous, though a homogeneous random walk must exist between the two random walks.
\begin{corollary}\label{corollary:slab_sandwich}
    For positive integer $k\geq2$, let $X$, $Y$, and $Z$ be irreducible random walks on the slab of thickness $k$ where $Y$ is homogeneous and $X\trianglelefteq Y\trianglelefteq Z$.
    If $Z$ is recurrent then $X$ is recurrent, and if $X$ is transient then $Z$ is transient.
\end{corollary}

\section{Random Walks on a Tree}\label{sec:trees}
Let $X$ and $Y$ denote two random walks on the same undirected tree $T=(V(T),E(T))$ rooted at $r$, and let $P$ and $Q$ denote their transition matrices.
For any vertex $v\neq r$, let $\pa(v)$ denote its parent, i.e., the next vertex on the unique path from $v$ to $r$.
We now formally define the two partial orders that we introduced in \cref{section:introduction}, which we denote by $\trianglelefteq$ and $\preceq$, respectively.
\begin{definition}
    We say $X\trianglelefteq Y$ if for all edges $(u,v)\in E(T)$, where say $u=\pa(v)$, we have $P(v,u)\geq Q(v,u)$.

    We say $X\preceq Y$ if for all edges $(u,v)\in E(T)$ with $u=\pa(v)$, we have $P(v,u)\geq Q(v,u)$ and $P(u,v) \leq Q(u,v)$.
\end{definition}
Note that because the outgoing transition probabilities from any state sum to 1, the second condition of $X\preceq Y$ implies the first condition.

\begin{remark}\label{remark:trees_classification}
    A classification of recurrence for random walks on trees, using the notion of logarithmic capacity, is known; see, for example, \cite[Theorem 6.13]{Woess}.
    However, this characterization is quite involved and technical, and it is not immediately clear what this result says concerning monotonicity of recurrence with respect to $\trianglelefteq$ and $\preceq$.
    To the best of our knowledge, our two results on monotonicity of recurrence for trees, \cref{theorem:tree_counterexample,theorem:tree_monotonicity}, are new.
\end{remark}

We first provide a counterexample to monotonicity of recurrence with respect to $\trianglelefteq$.
\begin{theorem}\label{theorem:tree_counterexample}
    There exists two random walks $X\trianglelefteq Y$ on the same tree $T$ such that $Y$ is positive recurrent and $X$ is transient.
    Moreover $X$ and $Y$ are elliptic, meaning all possible transitions have positive probabilities, and the defining inequalities for $X \trianglelefteq Y$ all strictly hold.
\end{theorem}
\begin{proof}
    Let $T$ be the binary tree, whose vertices we denote by $V(T)=\{L,R\}^*$, the set of binary strings on alphabet $\{L,R\}$, where the root is the empty string $\varepsilon$ and the parent of a vertex $v$ is obtained by removing its last character.
    We visually orient the tree so that the root is at the top, and the two children go down and to the left or down and to the right, depending on whether we added an $L$ or $R$, respectively.
    
    Let $P(\varepsilon,L)=Q(\varepsilon,L)=P(\varepsilon,R)=Q(\varepsilon,R)=\frac{1}{2}$.
    Fix $0<\delta<\frac{1}{7}$.
    We first construct $Y$ and show it is positive recurrent.
    Let all states in $L^*\setminus\{\varepsilon\}$, the non-empty strings consisting only of $L$, have outgoing transition probabilities of $2\delta$ going up, $\delta$ going left, and $1-3\delta$ going right.
    Let the rest of the vertices have outgoing probabilities of $1-2\delta$ going up, $\delta$ going left, and $\delta$ going right.
    Consider a random walk starting at $\varepsilon$.
    Every time we move right from a state in $L^*$, note that both downwards moves are equivalent, and because for all steps until we return to $L^*$, we have probability $1-2\delta>\frac{1}{2}$ probability of going up and $2\delta<\frac{1}{2}$ probability of going down, so with finite expected time we return to $L^*$ at the state we originally departed $L^*$ from.
    So now consider the infinite subsequence of this Markov chain that is obtained by restricting to state space $L^*$, i.e., if we move right we skip ahead until we return to $L^*$.
    Because each venture away from $L^*$ takes finite expected time to return, it suffices to show the expected number of steps for this subsequence to return to $\varepsilon$ is finite.
    This subsequence is a lazy Markov chain where the root goes left/down with probability $\frac{1}{2}$ and stays with probability $\frac{1}{2}$, and for all other states goes left/down with probability $\delta$, stays with probability $1-3\delta$, and goes up/right with probability $2\delta$.
    As $2\delta>\delta$, so we have mean drift towards the root, this Markov chain returns to the root in finite expected time, and thus $Y$ is positive recurrent.

    Now we construct $X$ and show it is transient.
    Let all states in $L^*\setminus\{\varepsilon\}$ have outgoing transition probabilities of $3\delta$ going up, $1-4\delta$ going left, and $\delta$ going right.
    Let the rest of the non-root vertices have outgoing probabilities of $1-\delta$ going up, $\delta/2$ going left, and $\delta/2$ going right.
    Note that $X\trianglelefteq Y$, with these inequalities strictly holding, and that both $X$ and $Y$ are elliptic.
    Similar to before, every time we move right from a state in $L^*$, with finite expected time we return to $L^*$, at the state we originally departed from.
    So it suffices to show that the infinite subsequence of this Markov chain that is obtained by restricting to $L^*$, which is again a lazy Markov chain, is transient.
    This Markov chain, for all non-root states, goes left/down with probability $1-4\delta$ and up with probability $3\delta$, so because $\delta<\frac{1}{7}$ it has mean drift away from the root, which means it is transient.
    This completes the proof.
\end{proof}
Next, we prove that monotonicity of recurrence holds on trees with respect to $\preceq$.
We also show monotonicity of positive recurrence, including the refinement that considers the expected return time.
\begin{theorem}\label{theorem:tree_monotonicity}
    Let $X$ and $Y$ be two irreducible random walks on the same tree $T$, rooted at $r$, such that $X\preceq Y$.
    Then if $Y$ is recurrent, so is $X$.
    Moreover, if $Y$ is positive recurrent, so that the expected return time $\E_{Y}[\tau_r]$ of $r$ is finite, then $X$ is positive recurrent, and $\E_{X}[\tau_r]\leq\E_{Y}[\tau_r]$.
\end{theorem}
\begin{proof}
    Our proof will use standard theory about the electrical network interpretation of (reversible) Markov chains; we refer readers to \cite{LyonsPeres} for a comprehensive overview.
    Recall each undirected edge $(x,y)$ is associated a conductance $c_{xy}\geq0$, where $c_x=\sum_y c_{xy}$ is the total conductance incident to a state $x$.
    In particular, we will use the fact that if $\sum_x c_x < \infty$, then the Markov chain is positive recurrent with stationary distribution $\pi$ given by $\pi_x \propto c_x$, and the expected return time of $x$ is $\frac{1}{\pi_x}\propto\frac{1}{c_x}$.
    We will also use the notion of effective conductance between a state and a set of states, as well as the characterization of recurrence, which when applied to our setting of a tree, means that a Markov chain is recurrent if and only if the effective conductance $c_n$ between the root and the level-$n$ vertices goes to 0 as $n\to\infty$.

    As $X\preceq Y$, the transition probabilities from the root are the same in $X$ and $Y$, so because the conductances of an electrical network corresponding to a given Markov chain are unique up to scaling, we can assign the same conductances $c_{rv}^{(X)}=c_{rv}^{(Y)}$ for each child $v$ of the root $r$, i.e., level-1 vertices $v$, for both the electrical networks corresponding to $X$ and $Y$.
    We claim that the conductances for the entire tree $T$, which we denote by $c_{vw}^{(X)}$ and $c_{vw}^{(Y)}$ for edge $(v,w)$ for $X$ and $Y$, respectively, satisfy $c_{vw}^{(X)}\leq c_{vw}^{(Y)}$ for all edges $(v,w)$.
    Note that for any level-1 vertex $v$ and a child $w$ of $v$, we have $c_{vw}^{(X)}=c_{rv}^{(X)}P(v,w)/P(v,r)$ and similarly for $Y$, where because $X\preceq Y$ we have $P(v,w)/P(v,r) \leq Q(v,w)/Q(v,r)$, and thus $c_{vw}^{(X)}\leq c_{vw}^{(Y)}$.
    Continuing this argument inductively by level demonstrates $c_{vw}^{(X)}\leq c_{vw}^{(Y)}$ for all edges $(v,w)$ of $T$.
    
    If $Y$ is recurrent, then the effective conductances $c_n^{(Y)}$ between the root and the level-$n$ vertices goes to 0 as $n\to\infty$.
    Rayleigh monotonicity then implies, because $c_{vw}^{(X)}\leq c_{vw}^{(Y)}$ for all edges $(v,w)$, that the corresponding effective conductance $c_n^{(X)}\leq c_n^{(Y)}$ for all $n$, and thus because (effective) conductance is always nonnegative, we have $c_n^{(X)}\to0$ as $n\to\infty$, which implies $X$ is recurrent as well.

    Now suppose $Y$ is positive recurrent.
    Letting $c_v^{(X)}=\sum_w c_{vw}^{(X)}$ and $c_v^{(Y)}=\sum_w c_{vw}^{(Y)}$ denote the total conductances incident to state $v$ for $X$ and $Y$, respectively, for all states $v$, our previous observation that the edgewise conductances of $X$ are less than or equal to that of $Y$ implies $c_v^{(X)}\leq c_v^{(Y)}$ for all $v$.
    Thus, because $Y$ is positive recurrent, we know $\sum_v c_v^{(Y)} < \infty$, and therefore $\sum_v c_v^{(X)}<\infty$ as well, and thus $X$ is positive recurrent.
    Moreover, to prove that $\E_X[\tau_r]\leq\E_y[\tau_r]$, it suffices to show $\pi_r \geq \sigma_r$, where $\pi$ is the stationary distribution for $X$ and $\sigma$ is the stationary distribution for $Y$.
    We have $\pi_r=c_r^{(X)}/\sum_v c_v^{(X)}$ while $\sigma_r=c_r^{(Y)}/\sum_v c_v^{(Y)}$, so because $c_r^{(X)}=c_r^{(Y)}$ by construction, as $c_{rv}^{(X)}=c_{rv}^{(Y)}$ for each level-1 vertex $v$, it suffices to show $\sum_v c_v^{(X)} \leq \sum_v c_v^{(Y)}$.
    But this immediately follows from the fact that $c_v^{(X)} \leq c_v^{(Y)}$ for all $v$, which completes the proof.
\end{proof}
\section{Parameter Space of Random Walks on Abelian Groups}\label{sec:parameter_space}
Consider a finitely generated abelian group $G$ and a non-empty finite subset $S\subseteq G$.
We consider the space of random walks on the Cayley graph $\Gamma(G,S)$.
In particular, let $A$ be the $(\abs{S}-1)$-dimensional standard simplex
\[ A = \{x\in\R^{\abs{S}}: x_1+x_2+\cdots+x_{\abs{S}}=1, x_i\geq0\text{ for all $i$}\}, \]
and enumerating $S=\{s_1,s_2,\dots,s_{\abs{S}}\}$, associate each $a\in A$ with the homogeneous random walk on state space $G$, where the transition probability from $g$ to $gs_i$ is $a_i$, for all $1\leq i \leq k$ and $g\in G$.
We will refer to such a random walk as a \textit{homogeneous random walk on $\Gamma(G,S)$}.
As $G$ is isomorphic to a group of the form $\Z^n\times H$ for some nonnegative integer $n$ and some finite abelian group $H$ (which can in turn be classified, but we do not need this), without loss of generality suppose $G=\Z^n\times H$, and write each $s_i$ as $(s_{i,1},s_{i,2},\dots,s_{i,n},s_i')\in\Z^n\times H$.
Note that the homogeneity of such a random walk implies all states have the same type, so it makes sense to refer to the entire random walk as recurrent, transient, positive recurrent, or null recurrent, even when the random walk is not irreducible, i.e., $S$ does not generate $G$.
Let $R\subseteq A$ be the subset of $A$ consisting of the points $a\in A$ that correspond to a recurrent random walk.
This section devotes itself to studying the properties of $R$.

We first provide a characterization of $R$.
This follows from standard theory on random walks on $\Z^n$, though to the best of our knowledge, this result has not been explicitly stated before in the literature.
\begin{proposition}\label{proposition:recurrence_classification}
    Using the notation above, let $R\subseteq A$ be the subset of recurrent homogeneous random walks on $\Gamma(G,S)$, where we view $G$ up to isomorphism as $\Z^n\times H$.
    Let $\pi:\Z^n\times H\to \Z^n$ be the projection $(x,h)\mapsto x$.
    For a given homogeneous random walk $a\in A$, let $T_a\subseteq\Z^n$ be given by
    \[ T_a=\pi(\{s_i\in S: a_i>0\}). \]
    Then $a\in R$ if and only if $\operatorname{span}(T_a)\subseteq\R^n$ has dimension at most 2 and
    \[ \sum_{i=1}^{\abs{S}} a_i \pi(s_i) = 0. \]
\end{proposition}
\begin{proof}
    Fix $a\in A$.
    Let $X=(X_0,X_1,\dots)$ denote a homogeneous random walk corresponding to $a$, starting at $X_0=0$.
    We first reduce to the case $G=\Z^n$ by observing that $X$ is recurrent if and only if $\pi(X_0),\pi(X_1),\dots$, which is a homogeneous random walk on $\Gamma(\Z^n,\pi(S))$ corresponding to the same parameters $a\in A$, is recurrent.
    To address the technicality of the case when $\pi|_S$ is not injective, note that the result immediately implies its natural generalization to the case when $S$ is a multiset.
    The forward direction is obvious; to see the reverse direction, note that each time $\pi(X_t)=0$, we know $X_t=(0,h)$ for some $h\in H$.
    For each possible such $h$, i.e., all $h\in H$ such that $(0,h)\in\operatorname{span}_{\N}(\{s_i\in S:a_i>0\})$ is a $\N$-linear combination of the elements of $\{s_i\in S:a_i>0\}$, because $h$ has finite order, there exists some nonnegative integer $k$ and some $\varepsilon>0$ such that we have $\P(X_k=0|X_0=(0,h))\geq\varepsilon$.
    As $H$ is finite, taking the maximum over $k$ and the minimum over $\varepsilon$ implies there exists some nonnegative integer $K$ and some $\delta>0$ such that with probability at least $\delta$, we have $X_{t+k}=0$ for some $0 \leq k \leq K$.
    Thus because the projected random walk returns to 0 infinitely many times, and each time we have at least $\delta$ probability that $X$ will return to 0 within $K$ steps, $X$ returns to 0 with probability 1.

    Now, assume $G=\Z^n$ (we can also assume $S$ is a set, though this is not needed nor used).
    Let $d$ be the dimension of $\operatorname{span}(T_a)$, and let $V=\operatorname{span}_{\Z}(T_a)\cong\Z^d$.
    By applying the isomorphism $V\to\Z^d$, without loss of generality we can thus view our random walk $X=(X_0,X_1,\dots)$ as a random walk on $\Z^d$, where $T_a$ generates $\Z^d$.
    
    We first show the two conditions, which we refer to as the dimension condition and the mean drift condition, respectively, are necessary.
    It is well-known (e.g., one can prove this via Hoeffding's inequality in a similar approach as we did in \cref{subsection:transient}) that if the mean drift condition does not hold, then the random walk is transient, and thus it is necessary.
    Thus, to show the dimension condition is necessary, suppose for sake of contradiction that $a\in R$ but $d\geq 3$.
    We know that $\operatorname{span}_{\Z}(T_a)=\Z^d$, and we claim that in fact $\operatorname{span}_{\N}(T_a)=\Z^d$ as well.
    As $a\in R$ implies the mean drift condition, we know that, letting $M\in\Z^{d\times\abs{T_a}}$ be the matrix whose columns are the elements of $T_a$, we have a solution $x\in\R^{\abs{T_a}}$ to $Mx=0$ whose entries are all strictly positive.
    By Gauss--Jordan elimination, we can find a basis for the nullspace $N(M)$ where all entries of all basis vectors are rational.
    Expressing $x$ as a linear combination of these vectors, by rounding the coefficients to sufficiently close rationals, we obtain a new $\widetilde x\in N(M)$ whose entries are all rational and strictly positive.
    Scaling this up, we can then assume $\widetilde x$ has positive integer entries, and thus $\operatorname{span}_{\Z}(T_a)\subseteq\operatorname{span}_{\N}(T_a)$ by adding a sufficiently large positive integer multiple of $\widetilde x$ to the integer coefficients defining the $\Z$-linear combination.

    As $\operatorname{span}_{\N}(T_a)=\Z^d$, our random walk is irreducible.
    As $d\geq3$, standard theory concerning irreducible homogeneous random walks on $\Z^d$ implies the random walk is transient, contradicting $a\in R$ (see, for example, \cite[Corollary 13.11]{Woess}, where using their notation of $\mu^{(k)}$ to denote the $k$-fold convolution of the transition law $\mu$, we have $d\geq3$ if and only if $\mu^{(k)}(0)$ is summable over $k$, or equivalently Green's function $G(0,0)<\infty$, which is well-known to be equivalent to transience).
    Thus, the dimension condition is necessary.

    Finally, we show the combination of these two conditions is sufficient.
    Suppose $d\leq 2$ and the mean drift condition holds.
    As before, the mean drift condition implies our random walk is irreducible, and so by \cite[Corollary 13.11]{Woess} again we know our random walk is recurrent.
\end{proof}
With this characterization, we can show that $R$ is always closed, both as a subset of $\R^{\abs{S}}$ and equivalently under the subspace topology on $A$, as $A$ is closed; this also implies $A\setminus R$ is open in the subspace topology on $A$.
\begin{theorem}\label{theorem:recurrence_closed}
    Using the notation above, let $R\subseteq A$ be the subset of recurrent homogeneous random walks on $\Gamma(G,S)$, for a finitely generated abelian group $G$ and non-empty finite subset $S\subseteq G$.
    Then $R$ is closed.
\end{theorem}
\begin{proof}
    Using the characterization from \cref{proposition:recurrence_classification}, note that there are finitely many values $\operatorname{span}(T_a)$ can take on, trivially at most $2^{\abs{S}}$.
    Thus, there are only finitely many such subspaces that have dimension at most 2.
    Pick such a subspace $V$, and note that requiring $\operatorname{span}(T_a)\subseteq V$ is equivalent to requiring $a_i=0$ for all $i$ such that $\pi(s_i)\not\in V$.
    The mean drift condition is also a linear equality in $a_1,a_2,\dots,a_{\abs{S}}$.
    Thus, recalling that $A$ is also defined by linear inequalities, the set of points that satisfy all these linear inequalities is closed.
    Taking the union over all such $V$ yields $R$ as a finite union of closed sets, which is closed.
\end{proof}
Next, we show that when $S$ is symmetric, i.e., closed under inverses, $R$ is path-connected.
\begin{theorem}\label{theorem:recurrence_connected_symmetric}
    Using the notation above, let $R\subseteq A$ be the subset of recurrent homogeneous random walks on $\Gamma(G,S)$, for a finitely generated abelian group $G$ and non-empty finite subset $S\subseteq G$.
    If $S$ is symmetric, then $R$ is path-connected.
\end{theorem}
\begin{proof}
    Consider two distinct points $a,b\in R$.
    By \cref{proposition:recurrence_classification}, we know that both satisfy the mean drift condition.
    Pick some $s_i\in S$ such that $a_i>0$ and some $s_j\in S$ such that $b_j>0$.
    As $S$ is symmetric, we know $-s_i,-s_j\in S$ as well.
    Let $a'\in A$ correspond to the homogeneous random walk with probability $\frac{1}{2}$ for each of $s_i$ and $-s_i$ (if $s_i=0$, then the random walk is the trivial walk with probability 1 of staying still).
    Using the notation of \cref{proposition:recurrence_classification}, as $\operatorname{span}(\{\pi(s_i),\pi(-s_i)\})\subseteq\operatorname{span}(T_a)$, we know that the closed line segment from $a$ to $a'$ lies in $R$.
    Similarly defining $b'\in A$ to correspond to the homogeneous random walk with probability $\frac{1}{2}$ for each of $s_j$ and $-s_j$, we have the closed line segment from $b$ to $b'$ lies in $R$.
    It then suffices to show $a'$ and $b'$ are path-connected, and we see that in fact, because $\operatorname{span}(\{\pi(s_i),\pi(s_j)\})$ is at most 2-dimensional, the closed line segment from $a'$ to $b'$ lies in $R$, so $a'$ and $b'$ are path-connected.
\end{proof}
Note that while $S$ being symmetric is sufficient for $R$ to be path-connected, it is not necessary; one could easily adapt the proof to show that an $S$ of the form $S=T\cup -cT$ for any positive integer $c$ will also yield a path-connected $R$.
At the same time, the following example, which takes inspiration from the trigonal bipyramidal molecular geometry structure in chemistry, shows that $R$ is not always path-connected.
\begin{example}
    Let $G=\Z^3$, and let $S=\{(0,0,1),(1,0,0),(-1,1,0),(-1,-1,0),(0,0,-1)\}$.
    Write $S=S_1\cup S_2$, where $S_1=\{(1,0,0),(-1,1,0),(-1,1-1,0)\}$ corresponds to the triangular base of the trigonal bipyramid, and $S_2=\{(0,0,1),(0,0,-1)\}$ corresponds to the apexes.
    Then by \cref{proposition:recurrence_classification}, we see that any $a\in R$ either has $T_a=S_1$ or $T_a=S_2$.
    The two components of $R$ corresponding to the cases $T_a=S_1$ and $T_a=S_2$ are both non-empty, and these two components do not intersect.
    Thus, these two components are not path-connected, and so $R$ is not path-connected.
    It is straightforward to verify that this is a minimum example of path-disconnected $R$, i.e., having the minimum values for $n$ and $\abs{S}$, namely 3 and 5, respectively.
\end{example}
A careful analysis of \cref{proposition:recurrence_classification}, in particular one that understands the intersection behavior between the various subspaces of the form $\operatorname{span}(T_a)$, should be able to completely classify when $R$ is path-connected in the general case when $S$ is not assumed to be symmetric; we leave this as an open question to the interested reader.

One can further ask when $R$ is not just path-connected, but in fact convex.
The following technical result characterizes when $R$ is convex.
\begin{theorem}\label{theorem:recurrence_convex}
    Using the notation of \cref{proposition:recurrence_classification}, $R$ is convex if and only if there exists a linear subspace $V\subseteq\R^n$ of dimension at most 2 such that for all $a\in R$, i.e., all $a\in A$ satisfying $\dim(\operatorname{span}(T_a))\leq2$ and $\sum_{i=1}^{\abs{S}}a_i \pi(s_i)=0$, we have $\operatorname{span}(T_a)\subseteq V$.
\end{theorem}
\begin{proof}
    The reverse direction is clear: if such a $V$ exists, then all convex combinations of points in $R$ will continue to satisfy the mean drift condition, as it is a linear equality, and because $\pi(s_i)\in V$ for all $i$ such that there exists $a\in R$ with $a_i>0$, we have $\operatorname{span}(T_a)\subseteq V$ has dimension at most 2 for any convex combination $a$ of points in $R$, so by \cref{proposition:recurrence_classification}, $a\in R$.

    To see the forward direction, suppose for sake of contradiction that there does not exist such a $V$.
    There must exist some $a\in R$ such that $\dim(\operatorname{span}(T_a))\geq1$, as otherwise $V=\{0\}$ suffices.
    And there must exist another $b\in R$ such that $\operatorname{span}(T_b)\not\subseteq\operatorname{span}(T_a)$, as otherwise $\operatorname{span}(T_a)$, which is at most 2-dimensional, would suffice as $V$.
    If $\dim(\operatorname{span}(T_a\cup T_b))>2$, then as $A$ is convex, letting $c=\frac{a+b}{2}\in A$, we have $\operatorname{span}(T_c)=\operatorname{span}(T_a\cup T_b)$ has dimension greater than 2, and thus $c\not\in R$, contradicting convexity of $R$.
    Otherwise, again letting $c=\frac{a+b}{2}\in A$, we have $\dim(\operatorname{span}(T_c))=2$, so $c\in R$.
    Then there must exist some $d\in R$ such that $\operatorname{span}(T_d)\not\subseteq\operatorname{span}(T_c)$, as otherwise $V=\operatorname{span}(T_c)$ would suffice.
    But then we similarly conclude $\frac{c+d}{2}\in A\setminus R$, because it yields a subspace of dimension greater than 2, contradicting convexity of $R$.
    Hence, such a $V$ must exist.
\end{proof}
Path-connectedness is a necessary condition for convexity, and as we saw previously, when $S$ is symmetric $R$ is always path-connected, but when $S$ is not symmetric the path-connectivity of $R$ is more subtle.
The following result implies a much simpler characterization for the convexity of $R$ when $S$ is symmetric: $R$ is convex if and only if $\dim(\operatorname{span}(\pi(S)))\leq2$, i.e., the generators essentially lie in a plane.
\begin{corollary}\label{corollary:recurrence_convex_symmetric}
    Using the notation of \cref{proposition:recurrence_classification}, if $\dim(\operatorname{span}(\pi(S)))\leq2$, then $R$ is convex, and if $\dim(\operatorname{span}(\pi(S)))\geq 3$ and $S$ is symmetric, then $R$ is not convex.
\end{corollary}
\begin{proof}
    The convex part immediately follows from \cref{theorem:recurrence_convex} by letting $V=\operatorname{span}(\pi(S))$.
    If $S$ is symmetric and $\dim(\operatorname{span}(\pi(S)))\geq3$, consider three linearly independent $\pi(s_i)$, $\pi(s_j)$, and $\pi(s_k)$.
    Then the homogeneous random walk with probability $\frac{1}{2}$ for each of $s_i$ and $-s_i$ is recurrent, and similarly for $s_j$ and $s_k$.
    However, the average $a$ of these three random walks has $\dim(\operatorname{span}(T_a))=3$, so it is not recurrent, and thus $R$ is not convex.
\end{proof}
Hence, in the symmetric case, $R$ is always path-connected but rarely convex, in the sense that the random walks must be essentially confined to a plane in order for $R$ to be convex.
Indeed, if we let $G=\Z^n$ and let $S$ consist of the elementary basis vectors and their inverses, we see $R$ is convex if and only if $n\leq 2$.

Another question of this topological flavor is whether $R$ is simply-connected or not, where recall convex implies simply-connected implies path-connected.
We leave this as an open question.
\begin{question}\label{question:recurrence_simply_connected}
    When is $R$ simply-connected?
\end{question}

Now, let us turn our attention instead to the set $R^c:=A\setminus R$ of transient homogeneous random walks.
\cref{theorem:recurrence_closed} tells us that $R^c$ is open in the subspace topology on $A$.

We note that $R^c$ is not always path-connected, even in the symmetric case.
Take $G=\Z$ and $S=\{1,-1\}$.
Then $A$ can be identified with $[0,1]$ via the projection $a\mapsto a_1$, where we see $R=\{1/2\}$, so $R^c = [0,1/2)\cup(1/2,1]$ is not path-connected.
However, we believe that $R^c$ should be path-connected aside from essentially similar one-dimensional examples, as typically $R$ will be a lower-dimensional object than $A$, and thus easy to avoid.
We leave this as an open question.
\begin{question}\label{question:transience_connected}
    When is $R^c$ path-connected?
\end{question}
We provide partial progress towards this question in the following result.
\begin{theorem}\label{theorem:transience_connected_3}
    Using the notation of \cref{proposition:recurrence_classification}, if there exists transient $a\in R^c$ satisfying the mean drift condition, i.e., $\dim(\operatorname{span}(T_a))\geq3$ and $\sum_{i=1}^{\abs{S}}a_i\pi(s_i)=0$, then $R^c$ is path-connected.
\end{theorem}
\begin{proof}
    We show that for any transient $a$ satisfying the mean drift condition, and any transient $b$ not satisfying the mean drift condition, the closed line segment between $a$ and $b$ is contained entirely in $R^c$.
    This proves the result, because to go from any transient $a$ to another transient $b$ along a path entirely in $R^c$, if exactly one satisfies the mean drift condition, a straight line segment suffices, and otherwise, either $a$ and $b$ both satisfy the mean drift condition or neither satisfies it.
    If neither satisfies it, the existence of a transient $c\in R^c$ satisfying the mean drift condition means we can follow two line segments, one from $a$ to $c$, then one from $c$ to $b$.
    If both satisfy it, then if there exists $c\in R^c$ not satisfying the mean drift condition, we similarly follow two line segments to and from $c$ to connect $a$ and $b$ along a path in $R^c$.
    Otherwise, all points in $A$ satisfy the mean drift condition, which means $\pi(s)=0$ for all $s\in S$, which by \cref{proposition:recurrence_classification} implies $R=A$ and thus $R^c=\emptyset$ is trivially path-connected.

    So consider any transient $a$ satisfying the mean drift condition and any transient $b$ not satisfying the mean drift condition.
    Then the closed line segment between $a$ and $b$ is contained in $R^c$ because the only point on this closed line segment that satisfies the mean drift condition is $a$, which we assumed is transient.
    This completes the proof.
\end{proof}
As partial progress towards \cref{question:transience_connected}, and in particular towards our belief that $R^c$ should be path-connected for examples on at least two dimensions, the following corollary of \cref{theorem:transience_connected_3} shows that symmetric examples on at least three dimensions have path-connected $R^c$.
\begin{corollary}\label{corollary:transience_connected_3}
    Using the notation of \cref{proposition:recurrence_classification}, if $S$ is symmetric and $\dim(\operatorname{span}(\pi(S)))\geq3$, then $R^c$ is path-connected.
\end{corollary}
\begin{proof}
    Similar to the proof of \cref{corollary:recurrence_convex_symmetric}, our assumptions imply the existence of an $a\in A$ with $\dim(\operatorname{span}(T_a))=3$ that satisfies the mean drift condition.
    Thus, by \cref{proposition:recurrence_classification}, $a$ is transient, so by \cref{theorem:transience_connected_3}, $R^c$ is path-connected.
\end{proof}
In the case $G=\Z^2$ and $S=\{e_1,-e_1,e_2,-e_2\}$ where $e_i$ denotes the $i$th elementary basis vector, applying \cref{proposition:recurrence_classification}, we can see that $R$ is characterized by the mean drift condition, namely $a_1=a_2$ and $a_3=a_4$, which, after adding in the equation $a_1+a_2+a_3+a_4=1$ that $A$ satisfies, defines a line going through (the interior of) the tetrahedron $A$.
As all points satisfying the mean drift condition are recurrent, \cref{theorem:transience_connected_3} does not apply, but nonetheless $R^c$ is path-connected, though is not simply-connected. 
This suggests the existence of an interesting regime where $R^c$ is path-connected but not simply-connected, which leads us to pose the following open question which we believe has a significantly different answer from that of \cref{question:transience_connected}.
\begin{question}\label{question:transience_simply_connected}
    When is $R^c$ simply-connected?
\end{question}
In general, one should not expect $R^c$ to be convex; we confirm this in the symmetric case.
\begin{proposition}\label{proposition:transience_not_convex_symmetric}
    Using the notation of \cref{proposition:recurrence_classification}, if $S$ is symmetric, then $R^c$ is not convex, except for the case when $\pi(s)=0$ for all $s\in S$, in which case $R^c = \emptyset$ is trivially convex.
\end{proposition}
\begin{proof}
    If $\pi(s)=0$ for all $s\in S$, then we have $R=A$ so $R^c=\emptyset$ is trivially convex.
    Otherwise, we have $\pi(s)\neq0$ for some $s\in S$, which implies $R^c\neq\emptyset$.
    As $S$ is symmetric and $\pi(s)\neq0$ for some $s\in S$, let $s_i=s$ and $s_j=-s$.
    Then $e_i,e_j\in R^c$ are both transient, while $\frac{e_i+e_j}{2}\in R$ is recurrent by \cref{proposition:recurrence_classification}, so $R^c$ is not convex.
\end{proof}
We leave the asymmetric case as an open question.
\begin{question}\label{question:transience_convex}
    When is $R^c$ convex?
\end{question}
\begin{remark}\label{remark:positive_recurrence_characterization}
    Note that if one wishes to study the subset $P\subseteq A$ corresponding to positive recurrent homogeneous random walks on $\Gamma(G,S)$, because irreducible homogeneous random walks with zero mean drift on $\Z^d$ are transient for $d\geq3$ and null recurrent for $d\in\{1,2\}$ (as the invariant measure, which is unique up to scaling, is identically 1 and this is not summable for $d\geq1$), using the notation of \cref{proposition:recurrence_classification}, we modify the condition to instead require $\dim(\operatorname{span}(T_a))=0$, i.e., $a_i>0$ implies $\pi(s_i)=0$.
    Note that this modified dimension condition trivially implies the mean drift condition.
    In other words, $P$ is simply the probability simplex over the indices $i$ such that $\pi(s_i)=0$.
    As a result, $P$ is closed and convex.
\end{remark}
Finally, we leave various generalizations of our setup as open questions.
\begin{question}\label{question:generalized_parameter_space}
    What can be said if we remove any of the following assumptions?
    \begin{itemize}
        \item $G$ is abelian.
        \item The random walk is homogeneous.
        \item $S$ is finite.
        \item $G$ is finitely generated.
    \end{itemize}
\end{question}

\section*{Acknowledgements}
We thank Jason Liu for many helpful discussions.
We also thank an anonymous referee for many helpful comments that improved the quality of our paper.
E.M. is supported in part by Vannevar Bush Faculty Fellowship ONR-N00014-20-1-2826 and Simons Investigator award 622132.

\bibliographystyle{amsinit}
\bibliography{ref}

\end{document}